\documentclass[11pt,english]{amsart}
\usepackage[latin9]{inputenc}
\usepackage[letterpaper]{geometry}
\usepackage{verbatim}
\usepackage{amstext}
\usepackage{amsthm}
\usepackage{amssymb}

\makeatletter
\numberwithin{equation}{section}
\numberwithin{figure}{section}
\theoremstyle{plain}
\newtheorem{thm}{\protect\theoremname}[section]
  \theoremstyle{remark}
  \newtheorem{rem}[thm]{\protect\remarkname}
  \theoremstyle{definition}
  \newtheorem{defn}[thm]{\protect\definitionname}
  \theoremstyle{definition}
  \newtheorem{example}[thm]{\protect\examplename}
  \theoremstyle{plain}
  \newtheorem{cor}[thm]{\protect\corollaryname}
  \theoremstyle{plain}
  \newtheorem{prop}[thm]{\protect\propositionname}
  \theoremstyle{plain}
  \newtheorem{lem}[thm]{\protect\lemmaname}
  \theoremstyle{plain}
  \newtheorem{conjecture}[thm]{\protect\conjecturename}

\usepackage[english]{babel}
\usepackage{amsfonts}
\usepackage[all]{xy}
\usepackage{tikz}
\usetikzlibrary{arrows}
\xyoption{arc}
\CompileMatrices

\title[Serre polynomials of character varieties]{Serre polynomials of $SL_n$- and $PGL_n$-character varieties of free groups}

\author[C. Florentino]{Carlos Florentino}

\address{Departamento de Matem\'{a}tica, Faculdade de Ci\^encias, Univ. de Lisboa, Campo Grande, Edf. C6, Lisbon, Portugal}

\email{caflorentino@ciencias.ulisboa.pt}

\author[A. Nozad]{Azizeh Nozad}

\address{School of Mathematics, Institute for Research in Fundamental Sciences (IPM), P.O.Box: 19395-5746, Tehran, Iran}

\email{anozad@ipm.ir}

\author[A. Zamora]{Alfonso Zamora}

\address{Departamento de Matem\'atica Aplicada a las TIC, ETSI Inform\'aticos, Universidad Polit\'ecnica de Madrid, Campus de Montegancedo,
28660, Madrid, Spain}

\email{alfonso.zamora@upm.es}

\thanks{This work was partially supported by CAMGSD and CMAFcIO (University of Lisbon), the projects QuantumG PTDC/MAT-PUR/30234/2017, FCT Portugal, a grant from IPM (Iran) and ICTP (Italy), and the projects MTM2016-79400-P and PID2019-108936GB-C21 by the Spanish government.}

\keywords{Character varieties, mixed Hodge structures, Serre polynomial, $E$-polynomial, representations of free groups}

\subjclass[2010]{14L30, 32S35, 14D20}

\renewcommand{\hom}{\mathrm{Hom}}
\def\quot{/\!\!/}
\def\sym{\mathsf{Sym}}

\newcommand{\gitq}{/\!\!/}

\newcommand{\pexp}{\operatorname{PExp}}
\newcommand{\plog}{\operatorname{PLog}}

\usepackage{babel}
\providecommand{\conjecturename}{Conjecture}
  \providecommand{\corollaryname}{Corollary}
  \providecommand{\definitionname}{Definition}
  \providecommand{\examplename}{Example}
  \providecommand{\lemmaname}{Lemma}
  \providecommand{\propositionname}{Proposition}
  \providecommand{\remarkname}{Remark}
\providecommand{\theoremname}{Theorem}

\usepackage{babel}
\providecommand{\conjecturename}{Conjecture}
  \providecommand{\corollaryname}{Corollary}
  \providecommand{\definitionname}{Definition}
  \providecommand{\examplename}{Example}
  \providecommand{\lemmaname}{Lemma}
  \providecommand{\propositionname}{Proposition}
  \providecommand{\remarkname}{Remark}
\providecommand{\theoremname}{Theorem}

\usepackage{babel}
\providecommand{\conjecturename}{Conjecture}
  \providecommand{\corollaryname}{Corollary}
  \providecommand{\definitionname}{Definition}
  \providecommand{\examplename}{Example}
  \providecommand{\lemmaname}{Lemma}
  \providecommand{\propositionname}{Proposition}
  \providecommand{\remarkname}{Remark}
\providecommand{\theoremname}{Theorem}

\makeatother

\usepackage{babel}
  \providecommand{\conjecturename}{Conjecture}
  \providecommand{\corollaryname}{Corollary}
  \providecommand{\definitionname}{Definition}
  \providecommand{\examplename}{Example}
  \providecommand{\lemmaname}{Lemma}
  \providecommand{\propositionname}{Proposition}
  \providecommand{\remarkname}{Remark}
\providecommand{\theoremname}{Theorem}

\begin{document}
\begin{abstract}
Let $G$ be a complex reductive group and $\mathcal{X}_{r}G$ denote
the $G$-character variety of the free group of rank $r$. Using geometric
methods, we prove that $E(\mathcal{X}_{r}SL_{n})=E(\mathcal{X}_{r}PGL_{n})$,
for any $n,r\in\mathbb{N}$, where $E(X)$ denotes the Serre (also
known as $E$-) polynomial of the complex quasi-projective variety
$X$, settling a conjecture of Lawton-Muñoz in \cite{LM}. The proof
involves the stratification by polystable type introduced in \cite{FNZ},
and shows moreover that the equality of $E$-polynomials holds for
every stratum and, in particular, for the irreducible stratum of $\mathcal{X}_{r}SL_{n}$
and $\mathcal{X}_{r}PGL_{n}$. We also present explicit computations
of these polynomials, and of the corresponding Euler characteristics,
based on our previous results and on formulas of Mozgovoy-Reineke
for $GL_{n}$-character varieties over finite fields. 
\end{abstract}

\maketitle

\section{Introduction}

Given a complex reductive algebraic group $G$, and a finitely presented
group $\Gamma$, the $G$-character variety of $\Gamma$ is the (affine)
geometric invariant theory (GIT) quotient 
\[
\mathcal{X}_{\Gamma}G=\hom(\Gamma,G)\quot G.
\]
The most well studied families of character varieties include the
cases when the group $\Gamma$ is the fundamental group of a Riemann
surface $\Sigma$, and its ``twisted'' variants. These varieties
correspond, via non-abelian Hodge theory, to certain moduli spaces
of $G$-Higgs bundles which play an important role in the quantum
field theory interpretation of the geometric Langlands correspondence
(see, for example \cite{Si}, \cite{KW}).

In the context of SYZ mirror symmetry \cite{SYZ}, the hyperkähler
nature of the Hitchin systems allowed a topological criterion for
mirror symmetry: as a coincidence between certain Hodge numbers of
moduli spaces of $G$-Higgs bundles over $\Sigma$, for Langlands
dual groups $G$ and $G^{L}$ (see \cite{Th}). The Hodge structure
of these moduli spaces is \emph{pure}, and topological mirror symmetry
has been established in the smooth/orbifold case for the pair of Langlands
dual groups $SL_{n}\equiv SL(n,\mathbb{C})$ and $PGL_{n}\equiv PGL(n,\mathbb{C})$
(see \cite{HT,GWZ}).

Such topological mirror symmetries are also expected for \emph{mixed}
Hodge structures on the cohomology of other classes of moduli spaces.
In the case of character varieties, whose Hodge structure is \emph{not
pure}, and which are generally, neither projective nor smooth, information
on the Hodge numbers is encoded in the Serre polynomial (also called
$E$-polynomial), which also provides interesting arithmetic properties
(see \cite{HRV}).

In this article, we consider character varieties of the free group
$\Gamma=F_{r}$ of rank $r\in\mathbb{N}$, and prove the following
equality. We denote the $E$-polynomial of a complex quasi-projective
variety $X$ by $E(X)$. 
\begin{thm}
\label{thm:main}Let $r\in\mathbb{N}$, and $\Gamma=F_{r}$. Then
\[
E(\mathcal{X}_{\Gamma}SL_{n})=E(\mathcal{X}_{\Gamma}PGL_{n}),\quad\forall n\in\mathbb{N}.
\]

\end{thm}
This result was proved for $n=2$ and 3 in Lawton-Muñoz \cite{LM},
using complex geometry methods; these computations increase substantially
for higher $n$, making it practically impossible to proceed explicitly.
In \cite[Remark 6]{LM} an arithmetic argument by Mozgovoy is mentioned,
showing the equality for all odd $n$ (by using a theorem of Katz
in \cite[Appendix]{HRV}, see also subsection 4.5), as further evidence
for the conjectured equality for all $n$, which we hereby confirm.

Denote by $\mathcal{X}_{\Gamma}^{irr}G\subset\mathcal{X}_{\Gamma}G$
the Zariski open subvariety of irreducible representation classes.
The following result, together with the stratification\textbackslash{}
by polystable type introduced in \cite{FNZ} for arbitrary $GL_{n}$-character
varieties, are the key ingredients in the proof of Theorem \ref{thm:main}. 
\begin{thm}
\label{thm:main-irred} Let $r\in\mathbb{N}$, and $\Gamma=F_{r}$.
Then 
\[
E(\mathcal{X}_{\Gamma}^{irr}SL_{n})=E(\mathcal{X}_{\Gamma}^{irr}PGL_{n}).\quad\forall n\in\mathbb{N}.
\]

\end{thm}
For the proof of this result, we use crucially the fact that free
group character varieties admit a strong deformation retraction to
analogous spaces of representations into $SU(n)$ and $PU(n)$, that
follow from \cite{FL1}.

The outline of the article is as follows. In Section 2, we review
standard facts on mixed Hodge structures and polynomial invariants,
including their equivariant versions. Section 3 deals with character
varieties in the context of affine GIT, which is enough to relate
the $E$-polynomials of $\mathcal{X}_{F_{r}}PGL_{n}$ with those of
$\mathcal{X}_{F_{r}}GL_{n}$. The polystable type stratification from
\cite{FNZ} is recalled in Section 4 for the $GL_{n}\equiv GL(n,\mathbb{C})$
case, and defined for the cases $SL_{n}$ and $PGL_{n}$; then we
provide the proof of the main theorem, assuming Theorem \ref{thm:main-irred}.
This second theorem is proved in Section 5, by examining the action
of $\hom(F_{r},Z)$ in the cohomology of $\hom(F_{r},SL_{n})$, where
$Z\cong\mathbb{Z}_{n}$ is the center of $SL_{n}$. Finally, Section
6 is devoted to describing a finite algorithm to obtain all the Serre
polynomials and the Euler characteristics of all the strata $\mathcal{X}_{F_{r}}^{[k]}G$,
for all partitions $[k]\in\mathcal{P}_{n}$ and $G=GL_{n}$, $SL_{n}$
or $PGL_{n}$. The algorithm uses formulae of Mozgovoy-Reineke for
$E(\mathcal{\mathcal{X}}_{F_{r}}^{irr}GL_{n})$, the irreducible $GL_{n}$
case \cite{MR}. Our main results were announced at the ISAAC conference
2019 \cite{FNSZ}.

\subsection{Comparison with related results in the literature }

For the benefit of the reader, we outline here some related previous
results, without pretending to be exhaustive, and summarize the main
novelties in our approach. For the surface group case (where $\Gamma=\pi_{1}(\Sigma)$
with $\Sigma$ a compact orientable surface, and related groups) the
calculations of Poincaré polynomials of $\mathcal{X}_{\Gamma}G$ started
with Hitchin and Gothen (for $G=SL_{n}$, $n=2,3$, see \cite{Hi,Got}),
and have been pursued more recently by García-Prada, Heinloth, Schmitt,
Hausel, Letellier, Mellit, Rodriguez-Villegas, Schiffmann and others,
who also considered the parabolic version of these character varieties
(see \cite{GPH,GPHS,HRV,Le,Me,Sc}).

Many of those recent results also use arithmetic methods: it is shown
that the number of points of the corresponding moduli space over finite
fields is given by a polynomial which, by Katz's theorem mentioned
above, coincides with the $E$-polynomial of $\mathcal{X}_{\Gamma}G$.
Then, in the smooth case, this allows the derivation of the Poincaré
polynomial from the $E$-polynomial.

The equality of \emph{stringy $E$-polynomials} (in the sense of Batyrev-Dais
\cite{BD}) of moduli spaces of $G$-Higgs bundles, for $SL_{n}$
and $PGL_{n}$, in the coprime case, was established by Hausel-Thaddeus,
for $n=2,3$, in \cite{HT}, and by Groechenig-Wyss-Ziegler for all
$n$ in \cite{GWZ}. More recently, Gothen-Oliveira considered the
parabolic version in \cite{GO}. In another direction, the \emph{full
Hodge-Deligne} polynomials of \emph{free abelian} group character
varieties were computed in \cite{FS}.

Most of the above results were obtained only in the smooth/orbifold
cases of the corresponding moduli spaces. On the other hand, for many
important classes of \emph{singular character varieties}, such as
\emph{free groups} or surface groups $\pi_{1}\Sigma$ \emph{without
twisting} (corresponding to degree zero bundles) explicitly computable
formulas for the $E$-polynomials are very hard to obtain. In the
articles of Lawton, Logares, Mart\'{i}nez, Muñoz and Newstead (by
using geometric methods, see \cite{LMN,Ma,MM} for surface groups
and \cite{LM} for the free group) and of Cavazos-Lawton and Baraglia-Hekmati
(\cite{CL} and \cite{BH}, using arithmetic methods), the Serre polynomials
were computed for several character varieties, with $G=GL_{n}$, $SL_{n}$
and $PGL_{n}$ for small values of $n$, but the computations quickly
become intractable for $n$ higher than 3. Recently, $SL_{n}$-character
varieties of surface groups were found to give rise to a Lax monoidal
topological quantum field theory (see \cite{GPLM}). In the case of
free group character varieties, the $E$-polynomials were obtained
in \cite{MR}, by point counting over finite fields, for all $GL_{n}$,
and in \cite{BH} for $SL_{n}$ with $n=2,3$.

Our present results and methods differ from the previous literature
in the following aspects. We consider the singular character variety
of the free group, for $G=SL_{n}$ and \emph{every} $n\in\mathbb{N}$,
and consider the standard compactly supported $E$-polynomial (not
the stringy one). We extend the stratification by polystable type
(described in \cite{FNZ} for $GL_{n}$) to the cases of $PGL_{n}$
and $SL_{n}$, and carefully examine the action of the centre of $SL_{n}$
on the cohomology of several spaces, to prove the main results. Our
Theorems \ref{thm:main} and \ref{thm:main-irred} do not use point
counting methods over finite fields, and we only use the formulas
of Mozgovoy-Reineke for $GL_{n}$-character varieties of $F_{r}$
(see \cite{MR}) in the last section, to provide explicit formulas
for Serre polynomials and Euler characteristics for all $\mathcal{X}_{F_{r}}SL_{n}$
and all polystable strata.

\subsection*{Acknowledgements}

We would like to thank many interesting and useful conversations with
several colleagues on topics around $E$-polynomials in particular
G. Granja, S. Lawton, M. Logares, V. Muñoz, A. Oliveira, F. Rodriguez-Villegas,
J. Silva, and the anonymous referee for useful suggestions. We also
thank the organizers of the Special Session on Geometry of Representation
Spaces of the AMS Joint meeting 2019 (Baltimore, USA), and the Special
Session on Complex Geometry of the Conference ISAAC 2019 (Aveiro,
Portugal) where these results were presented.

\section{Mixed Hodge structures and Serre polynomials}

This section recalls some standard facts on mixed Hodge structures
and polynomial invariants, introducing terminology and notation that
will be used throughout.

Let $X$ be a quasi-projective variety over $\mathbb{C}$, of complex
dimension $d$, which may be singular, not complete, and/or not irreducible.
Following Deligne (c.f. \cite{De}, \cite{PS}), the compactly supported
cohomology $H_{c}^{*}(X):=H_{c}^{*}(X,\mathbb{C})$ is endowed with
a mixed Hodge structure. Denote the corresponding mixed Hodge numbers
by 
\[
h^{k,p,q}(X)=\dim_{\mathbb{C}}H_{c}^{k,p,q}(X)\in\mathbb{N}_{0},
\]
for $k\in\{0,\cdots,2d\}$ and $p,q\in\{0,\cdots,k\}$. We say that
$(p,q)$ are $k$-weights of $X$, when $h^{k,p,q}\neq0$.

Mixed Hodge numbers verify $h^{k,p,q}(X)=h^{k,q,p}(X)$, and $\dim_{\mathbb{C}}H_{c}^{k}(X)=\sum_{p,q}h^{k,p,q}(X)$,
so they provide the (compactly supported) Betti numbers, which are
easily translated to the usual Betti numbers, in the smooth case,
by Poincaré duality.

Hodge numbers yield the so-called mixed Hodge polynomial of $X$,
\begin{equation}
\mu(X;\,t,u,v):=\sum_{k,p,q}h^{k,p,q}(X)\ t^{k}u^{p}v^{q}\in\mathbb{\mathbb{N}}_{0}[t,u,v],\label{eq:mu}
\end{equation}
on three variables. The mixed Hodge polynomial specializes to the
(compactly supported) Poincaré polynomial by setting $u=v=1$, $P_{t}^{c}(X)=\mu(X;\,t,1,1).$
Again, this provides the usual Poincaré polynomial in the smooth situation.

By substituting $t=-1$, mixed Hodge polynomials become a very useful
generalization of the Euler characteristic, called the \emph{Serre
polynomial} or $E$\emph{-polynomial} of $X$: 
\[
E(X;\,u,v):=\sum_{k,p,q}(-1)^{k}h^{k,p,q}(X)\ u^{p}v^{q}\in\mathbb{Z}[u,v].
\]
From the $E$-polynomial we can compute the Euler characteristic of
$X$ as $\chi(X)=E(X;\,1,1)=\mu(X;\,-1,1,1)$ (recall that the compactly
supported Euler characteristic equals the usual one for complex quasi-projective
varieties).

When $X$ is smooth and projective, its Hodge structure is \emph{pure}
(for each $k$, the only $k$-weights are of the form $(p,k-p)$ with
$p\in\{0,\cdots,k\}$), and the $E$-polynomial actually determines
the Poincaré polynomial $P_{t}=P_{t}^{c}$.

In the present article, we deal with varieties which are neither smooth
nor complete, but often are of \emph{Hodge-Tate type} (also called
\emph{balanced} type), for which all the $k$-weights are of the form
$(p,p)$ with $p\in\{0,\cdots,k\}$. This restriction on weights holds
for complex (affine) algebraic groups (see, for example \cite{DL,Jo})
and smooth toric varieties, among others classes.

Poincaré, mixed Hodge and Serre polynomials satisfy a multiplicative
property with respect to Cartesian products $\mu(X\times Y)=\mu(X)\mu(Y)$
(coming from Künneth theorem, see \cite{PS}). The big computational
advantage of $E(X)$, as compared to $\mu(X)$ or $P^{c}(X)$ is that
it also satisfies an \emph{additive property with respect to stratifications}
by locally closed (in the Zariski topology) strata: whenever $X$
has a closed subvariety $Z\subset X$ we have 
\[
E(X)=E(Z)+E(X\setminus Z),
\]
(see, eg. \cite{PS}) which generalizes the well known analogous statement
for $\chi$. 
\begin{rem}
The terminology for $E$-polynomials is not standard, and some authors
refer to them as \emph{Hodge-Deligne} (mostly when dealing with pure
Hodge structures), and others as \emph{virtual Poincaré or Serre polynomials}.
After a literature review process, we are using \emph{Serre polynomial},
following the oldest references (see, for example \cite{GL}), and
paying tribute to the ideas of Serre on the role of virtual dimensions
for additivity in the weight filtration on mixed Hodge structures
(see \cite{To}). 
\end{rem}

\subsection{Equivariant Serre polynomials and fibrations}

We will also need the multiplicative property of the $E$-polynomial
under certain algebraic fibrations, and more generally, for the \emph{equivariant
$E$-polynomial}, when these fibrations are acted by a finite group. 
\begin{defn}
\label{def:alg-W-fibration}Let $\pi:X\to B$ be a morphism of quasi-projective
varieties and $W$ be a finite group acting algebraically on $X$,
and preserving the fibers of $\pi$. Assume also that all fibers $\pi^{-1}(b)$,
$b\in B$, are $W$-isomorphic to a given variety $F$. In this situation,
we call 
\begin{equation}
F\to X\to B\label{eq:W-fibration}
\end{equation}
an \emph{algebraic $W$-fibration.} 
\end{defn}
Given an algebraic $W$-fibration $F\to X\to B$, the mixed Hodge
structures on the (compactly supported) cohomology groups of $F$
and $X$ are representations of $W$, which we denote by $[H_{c}^{k,p,q}(F)]$
and similarly for $X$. This allows us to define the $W$-equivariant
$E$-polynomials: 
\[
E^{W}(X;\,u,v)=\sum_{k,p,q}(-1)^{k}[H_{c}^{k,p,q}(X)]\ u^{p}v^{q}\in R(W)[u,v],
\]
where $R(W)$ denotes the representation ring of $W$. By taking the
dimensions of the fixed subspaces $[H_{c}^{k,p,q}(X)]^{W}$, under
$W$, we recover the $E$-polynomial of the quotient variety $X/W$
(see \cite{DL}, \cite{FS}).

We have the following fundamental result. 
\begin{thm}
\label{thm:W-fibration} \cite{DL,LMN} Let $W$ be a finite group
and consider an algebraic $W$-fibration as above: 
\[
F\to X\to B.
\]
Assume either:\\
 (i) the fibration is locally trivial in the Zariski topology of $B$,
or\\
 (ii) $F$, $X$ and $B$ are smooth, the fibration is locally trivial
in the complex analytic topology, and $\pi_{1}(B)$ acts trivially
on $H_{c}^{*}(F)$ (ie, the monodromy is trivial), or\\
 (iii) $X$, $B$ are smooth and $F$ is a complex connected Lie group.\\
 Then 
\[
E^{W}(X)=E^{W}(F)\cdot E(B).
\]
Moreover, if $F$ and $B$ are of Hodge-Tate type, then so is $X$.\end{thm}
\begin{proof}
See \cite[Thm. 6.1]{DL} or \cite{FS}. For (iii) and for the result
on Hodge-Tate (balanced) type see also \cite{LMN}. 
\end{proof}
Of course, when $W=1$ or the $W$ action is trivial, the hypothesis
imply the product formula $E(X)=E(F)\cdot E(B)$. 
\begin{example}
The condition of $F$ being connected in (iii) is necessary for the
multiplicative property to hold, even in the case of trivial $W$.
Indeed, the action of $F=\mathbb{Z}_{2}$ on $X=\mathbb{P}^{1}\times\mathbb{P}^{1}$
by permuting the entries provides a fibration onto $\mathbb{P}^{2}$,
the second symmetric power of $\mathbb{P}^{1}$: 
\[
\mathbb{Z}_{2}\rightarrow\mathbb{P}^{1}\times\mathbb{P}^{1}\rightarrow\sym^{2}(\mathbb{P}^{1})=\mathbb{P}^{2}\;,
\]
By elementary methods, the $E$-polynomials are $E(\mathbb{Z}_{2})=2$,
$E(\mathbb{P}^{1}\times\mathbb{P}^{1})=(1+uv)^{2}$ and $E(\mathbb{P}^{2})=1+uv+u^{2}v^{2}$,
which do not satisfy the multiplicative property. One can also check
that this algebraic fibration is not locally trivial in the Zariski
topology, but only in the analytic (strong) topology. 
\end{example}

\subsection{Special fibrations}

The notion of special group according to Serre and Grothendieck (c.f.
\cite{Gr,Se}) provides a useful criterion for applying Theorem \ref{thm:W-fibration}
to some algebraic fibrations.

By definition, a \emph{special group} is an algebraic group $H$ such
that every principal $H$-bundle is locally trivial in the Zariski
topology (c.f. \cite[p.11]{Gr}). In this context, a principal $H$-bundle
$\pi:X\to B$, will be called a \emph{special fibration}, whenever
$H$ is special. If, furthermore, a finite group $W$ acts on $X$
as in Definition \ref{def:alg-W-fibration}, we call $\pi$ a \emph{special
$W$-fibration}. Thus, Given a special $W$-fibration $H\to X\to B$,
the following is an immediate consequence of Theorem \ref{thm:W-fibration}(i). 
\begin{cor}
\label{cor:special-fibration} Let $X\to B$ be a special $W$-fibration
with fiber isomorphic to $H$. Then 
\[
E^{W}(X)=E^{W}(H)\cdot E(B).
\]

\end{cor}
Now, let $G$ be a connected complex algebraic reductive group, with
center $Z\subset G$, and denote by $PG=G/Z$ its adjoint group. In
case $Z$ is connected, we obtain a simple relation between the Serre
polynomials of $G$, $Z$ and $PG$. 
\begin{prop}
\label{prop:PG-fibration}Let $G$ be a complex reductive group with
connected center $Z$. Then $E(G)=E(PG)\cdot E(Z)$. \end{prop}
\begin{proof}
Since $PG$ is defined to be a quotient by a subgroup, the following
fibration 
\begin{equation}
Z\to G\to PG,\label{eq:ZG-G-PG}
\end{equation}
is a principal $Z$-bundle. Since $G$ is complex, its center $Z$
is the product of a group of multiplicative type and a unipotent subgroup.
Given that $G$ is reductive, the unipotent subgroup is trivial. Since
$Z$ is connected by hypothesis, $Z$ is a torus, that is $Z\cong(\mathbb{C}^{\ast})^{l}$,
$l=\dim Z$. Since a torus is a \textit{\emph{special}}\emph{ }group
by \cite{Se}, \eqref{eq:ZG-G-PG} is a special fibration and the
result follows from Corollary \ref{cor:special-fibration} (with $W$
being the trivial group).\end{proof}
\begin{example}
\label{exa:PGL-section} Since $E(\mathbb{C}^{*})=uv-1$, the formula
$E(GL_{n})=(uv-1)E(PGL_{n})$ follows from the special fibration 
\[
\mathbb{C}^{*}\to GL_{n}\to PGL_{n},
\]
which will be very important later on, and can be directly shown to
be locally trivial in the Zariski topology. By contrast, the determinant
fibration 
\[
SL_{n}\to GL_{n}\stackrel{\det}{\to}\mathbb{C}^{*}
\]
is not trivial in the Zariski topology, although we still have $E(GL_{n})=(uv-1)E(SL_{n})$. \end{example}
\begin{rem}
\label{rem:Hodge-Tate-open}As mentioned before, complex (affine)
algebraic groups $G$ are of Hodge-Tate type. Hence, their mixed Hodge
polynomials $\mu(G;t,u,v)$ reduce to a two variable polynomial (with
variables $t$ and $uv$), and their $E$-polynomials depend only
on the product $uv$. On the other hand, given a variety $X$ whose
$E$-polynomial is a function of only $x=uv$, it is not necessarily
true that $X$ is of Hodge-Tate type. For example, if we had $\mu(X;t,u,v)=1+tu+t^{2}u+t^{2}uv$,
then $X$ could not be of Hodge-Tate type, but $E(X)=1+uv=E(\mathbb{P}^{1})$
is a one-variable polynomial in $x:=uv$. We strongly believe that
the character varieties studied in this paper are also of Hodge-Tate
(balanced) type. Indeed, the methods in \cite{LMN,LM} show this is
the case for $n=2$ and 3. However, as far as we know, the general
case seems to be an open problem (see also Conjecture \ref{conj:Langlands-dual}). 
\end{rem}


\section{Character Varieties for a group and its adjoint}

As before, let $G$ be a connected complex reductive algebraic group
with center $Z$. In this section, we recall some aspects of character
varieties and their construction via affine GIT (Geometric Invariant
Theory). This is applied to provide a simple relation between the
Serre polynomials of $\mathcal{X}_{\Gamma}GL_{n}$ and of $\mathcal{X}_{\Gamma}PGL_{n}$,
for the free group $\Gamma=F_{r}$.

\subsection{Character varieties as GIT quotients.}

Let $\Gamma$ be a finitely presented group, and denote by 
\[
\mathcal{R}_{\Gamma}G=\hom(\Gamma,G),
\]
the algebraic variety of representations of $\Gamma$ in $G$, where
$\rho\in\mathcal{R}_{\Gamma}G$ is defined by the image of the generators
of $\Gamma$, $\rho(\gamma)$, satisfying the relations in $\Gamma$.
The moduli space of representations of $\Gamma$ into $G$ is the
$G$-character variety of $\Gamma$ 
\[
\mathcal{X}_{\Gamma}G:=\hom(\Gamma,G)\quot G,
\]
defined as the affine GIT quotient under the algebraic action of $G$
on $\mathcal{R}_{\Gamma}G$ by conjugation of representations. Given
that a GIT quotient identifies those orbits whose closure has a non-empty
intersection, we describe the quotient by the unique closed point
in each equivalence class, called the \emph{polystable representations},
which we define below. Given a representation $\rho\in\mathcal{R}_{\Gamma}G$,
denote by $Z_{\rho}$ the centralizer of $\rho(\Gamma)$ inside $G$.
Note that $ZG\subset Z_{\rho}$, since the center commutes with any
representation. Let us call the subgroup $G_{\mathcal{R}_{\Gamma}G}:=\bigcap_{\rho\in\mathcal{R}_{\Gamma}G}Z_{\rho}$
the \emph{center of the action}, since it acts trivially, and $G/G_{\mathcal{R}_{\Gamma}G}$
acts effectively on $\mathcal{R}_{\Gamma}G$. Denote by $\psi_{\rho}$
the (effective) orbit map through $\rho$: 
\begin{eqnarray*}
\psi_{\rho}:G/G_{\mathcal{R}_{\Gamma}G} & \to & \mathcal{R}_{\Gamma}G\\
g & \mapsto & g\cdot\rho
\end{eqnarray*}

\begin{defn}
\label{def:center-polystable}In the situation above, we say that
$\rho\in\mathcal{R}_{\Gamma}G$ is \emph{polystable} if the orbit
$G\cdot\rho:=\{g\rho g^{-1}:\ g\in G\}$ is closed (Zariski) in $\mathcal{R}_{\Gamma}G$.
We say that $\rho\in\mathcal{R}_{\Gamma}G$ is \emph{stable} if it
is polystable and $\psi_{\rho}$ is a proper map. 
\end{defn}
It can be shown that the subset of polystable representations in $\mathcal{R}_{\Gamma}G$,
denoted by $\mathcal{R}_{\Gamma}^{ps}G\subset\mathcal{R}_{\Gamma}G$,
is stratified by locally-closed subvarieties, yielding a quotient
which can be identified with the affine GIT quotient (see \cite{FL2}):
\[
\mathcal{X}_{\Gamma}G=\mathcal{R}_{\Gamma}G\quot G\cong\mathcal{R}_{\Gamma}^{ps}G/G.
\]

\begin{defn}
We say that $\rho\in\mathcal{R}_{\Gamma}G$ is irreducible if $\rho$
is polystable and $Z_{\rho}$ is a finite extension of $ZG$, and
call $\rho$ a good representation if it is irreducible and $Z_{\rho}=ZG$. 
\end{defn}
Denote by $\mathcal{R}_{\Gamma}^{irr}G\subset\mathcal{R}_{\Gamma}^{ps}G$
and by $\mathcal{R}_{\Gamma}^{g}G$ the subset of irreducible and
good representations, respectively, and set $\mathcal{X}_{\Gamma}^{irr}G:=\mathcal{R}_{\Gamma}^{irr}G/G$
and $\mathcal{X}_{\Gamma}^{g}G:=\mathcal{R}_{\Gamma}^{g}G/G$. It
is known that $\mathcal{R}_{\Gamma}^{irr}G$ is a quasi-projective
variety, since it is a Zariski open subset of $\mathcal{R}_{\Gamma}G$.
The character variety of good representations is a smooth algebraic
variety, by \cite{Sik}. For further details, we refer the reader
to \cite{FLR}, \cite{GLR} or \cite[Section 3]{FNZ}.

In the case when $G=GL_{n}$ or $G=SL_{n}$, Schur's lemma easily
implies the equivalence between good and irreducible representations. 
\begin{lem}
\label{lem:Schur}Let $G=GL_{n}$ or $G=SL_{n}$ and $\rho\in\hom(\Gamma,G)$
be polystable. Then $\rho$ is a good representation if and only if
it is irreducible. In particular $\mathcal{X}_{\Gamma}^{g}G=\mathcal{X}_{\Gamma}^{irr}G$.\end{lem}
\begin{proof}
By definition, a good representation is irreducible, and the converse
follows from \cite{FL3}; we can also proceed as follows. First let
$G=GL_{n}$. If $\rho$ is irreducible, Schur's lemma states that
any $g\in G$ commuting with every $\rho(\gamma)$, $\gamma\in\Gamma$,
is central, so that $Z_{\rho}=ZGL_{n}=\mathbb{C}^{*}$. Therefore,
$\rho$ is good.

Now let $G=SL_{n}$. For any representation $\rho:\Gamma\to SL_{n}$
we get a representation of $GL_{n}$ by composition $i\circ\rho:\Gamma\to SL_{n}\overset{i}{\hookrightarrow}GL_{n}$.
Since $SL_{n}$ is the derived group of $GL_{n}$, we have 
\[
\frac{Z_{i\circ\rho}}{ZGL_{n}}=\frac{Z_{i\circ\rho}\cap SL_{n}}{ZGL_{n}\cap SL_{n}}\text{, and }Z_{\rho}=Z_{i\circ\rho}\cap SL_{n}.
\]
Hence, if $\rho$ is an irreducible representation of $SL_{n}$ then
$i\circ\rho$ is an irreducible representation of $GL_{n}$ and hence,
by the previous argument, $i\circ\rho$ is a good $GL_{n}$-representation.
Therefore by using the fact that $Z_{\rho}=Z_{i\circ\rho}\cap SL_{n}$
we get $Z_{\rho}=ZSL_{n}$ which says that $\rho$ is a good representation
into $SL_{n}$. \end{proof}
\begin{prop}
\label{prop:good}Let $G=G_{1}\times G_{2}$, be a product of two
reductive groups. We have the following isomorphism of smooth algebraic
varieties: 
\[
\mathcal{X}_{\Gamma}^{g}(G)=\mathcal{X}_{\Gamma}^{g}(G_{1})\times\mathcal{X}_{\Gamma}^{g}(G_{2}).
\]
\end{prop}
\begin{proof}
We have $\mathcal{X}_{\Gamma}(G)=\mathcal{X}_{\Gamma}(G_{1})\times\mathcal{X}_{\Gamma}(G_{2})$
as algebraic varieties, since conjugation by $G=G_{1}\times G_{2}$
works independently on each factor of (cf. also \cite{FL1}): 
\[
\hom(\Gamma,G)=\hom(\Gamma,G_{1})\times\hom(\Gamma,G_{2}).
\]
Now suppose that a given representation $\rho\in\hom(\Gamma,G)$ has
a trivial stabilizer (the center of $G$). Then it is a product of
the centers of $G_{1}$ and of $G_{2}$, so $\rho$ is of the form
$\rho=(\rho_{1},\rho_{2})$ with good factors $\rho_{i}\in\hom(\Gamma,G_{i})$,
$i=1,2$. The converse is also clear, so the above isomorphism restricts
to the good loci. 
\end{proof}
Put together, the previous 2 statements show that, for $GL_{n}$ and
for $SL_{n}$-character varieties, the good locus coincides with both
the irreducible locus and with the stable locus, and these are multiplicative,
under products of reductive groups.

\subsection{$GL_{n}$- and $PGL_{n}$-character varieties of the free group}

When considering $\Gamma=F_{r}$, the free group of rank $r$, we
use the simplified notations $\mathcal{R}_{r}G:=\hom(F_{r},G)$ and
$\mathcal{X}_{r}G:=\hom(F_{r},G)\gitq G$, for the $G$-representation
and $G$-character varieties, respectively.

There is a natural action of $\mathcal{R}_{r}\mathbb{C}^{*}:=\hom(F_{r},\mathbb{C}^{*})\cong(\mathbb{C}^{*})^{r}$
on $\mathcal{R}_{r}GL_{n}$, which, in terms of the fixed generators
$\gamma_{1},\cdots,\gamma_{r}$ of $F_{r}$ is given by scalar multiplication
of representations: 
\begin{equation}
\sigma\cdot\rho(\gamma_{i}):=\sigma(\gamma_{i})\rho(\gamma_{i}),\quad\sigma\in\mathcal{R}_{r}\mathbb{C}^{*},\ \rho\in\mathcal{R}_{r}GL_{n}.\label{eq:C-GLnaction}
\end{equation}
The quotient of this action is the representation variety for $PGL_{n}$,
$\mathcal{R}_{r}PGL_{n}$. Since the central $\mathbb{C}^{*}$ action
commutes with conjugation of representations, this sequence descends
to the character varieties 
\begin{equation}
(\mathbb{C}^{\ast})^{r}\rightarrow\mathcal{X}_{r}GL_{n}\rightarrow\mathcal{X}_{r}PGL_{n}.\label{eq:GLn-PGLn-fibration}
\end{equation}
Note that, because $\Gamma=F_{r}$, all $PGL_{n}$ representations
can be lifted to $GL_{n}$. The following generalization to other
groups $G$ and their adjoints $PG$ is immediate from Corollary \ref{cor:special-fibration}. 
\begin{prop}
\label{prop:PG-fibration-X} Let $G$ be a connected reductive group
with connected center $Z$. Then, the natural quotient map $\mathcal{X}_{r}G\to\mathcal{X}_{r}PG$
is a special fibration, with fiber $\mathcal{X}_{r}Z\cong Z^{r}$.
Hence, 
\[
E(\mathcal{X}_{r}G)=(uv-1)^{rl}\,E(\mathcal{X}_{r}PG),
\]
with $l=\dim Z$. In particular, for $G=GL_{n}$ we get $E(\mathcal{X}_{r}GL_{n})=(uv-1)^{r}\,E(\mathcal{X}_{r}PGL_{n})$. 
\end{prop}

\section{The strictly polystable case}

\label{section:stratifications}In this section and the next one,
we prove Theorem \ref{thm:main}: the equality of the Serre polynomials
of $\mathcal{X}_{r}SL_{n}$ and $\mathcal{X}_{r}PGL_{n}$. If we tried
to imitate the fibration methods of Section 3, we would consider the
quotient morphism $SL_{n}\to PGL_{n}$, with kernel $\mathbb{Z}_{n}$,
the center of $SL_{n}$, to obtain a fibration of character varieties:
\[
\mathbb{Z}_{n}^{r}\to\mathcal{X}_{r}SL_{n}\to\mathcal{X}_{r}PGL_{n}.
\]
However, since the fiber is not connected, we cannot apply Corollary
\ref{cor:special-fibration}; instead we proceed by stratifying this
fibration by polystable type, in analogy to the stratification used
in \cite{FNZ} (recalled below) and examine the $\mathbb{Z}_{n}^{r}$
action on the cohomology of each individual stratum. From now on,
our $E$-polynomials will only depend on a single variable; to emphasize
this property, we adopt the substitution $x\equiv uv$ and use the
notation: 
\begin{equation}
e(X):=E(X;\,\sqrt{x},\sqrt{x})\in\mathbb{Z}[x].\label{eq:small-e}
\end{equation}

\subsection{Stratifications by polystable type}

Any character variety admits a stratification by the dimension of
the stabilizer of a given representation. When dealing with the general
linear group $GL_{n}$ as well as the related groups $SL_{n}$ and
$PGL_{n}$, there is a more refined stratification which gives a lot
more information on the corresponding character varieties $\mathcal{X}_{\Gamma}G$.
In this subsection, we recall this stratification, following \cite[Section 4.1]{FNZ},
and describe its analogous versions for $SL_{n}$ and $PGL_{n}$.

A \emph{partition} of $n\in\mathbb{N}$ is denoted by $[k]=[1^{k_{1}}\cdots j^{k_{j}}\cdots n^{k_{n}}]$
where the exponent $k_{j}$ means that $[k]$ has $k_{j}\geq0$ parts
of size $j\in\{1,\cdots,n\}$, with $n=\sum_{j=1}^{n}j\,k_{j}$. The
\emph{length} of the partition is given by the sum of the exponents
$|[k]|:=\sum k_{j}$ and call $\mathcal{P}_{n}$ the set of all partitions
of $n\in\mathbb{N}$. As an example, $[1^{2}\;2\;4]\in\mathcal{P}_{8}$
is the partition $8=1\cdot2+2+4$, with length equal to $4$. 
\begin{defn}
Let $n\in\mathbb{N}$, $[k]\in\mathcal{P}_{n}$ and $\Gamma$ be a
finitely presented group. We say that $\rho\in\mathcal{R}_{\Gamma}GL_{n}=\hom(\Gamma,GL_{n})$
is \emph{$[k]$-polystable} if $\rho$ is conjugated to $\bigoplus_{j=1}^{n}\rho_{j}$
where each $\rho_{j}$ is, in turn, a direct sum of $k_{j}>0$ \emph{irreducible}
representations of $\mathcal{R}_{\Gamma}(GL_{j})$, for $1\leq j\leq n$
(by convention, if some $k_{j}=0$, then $\rho_{j}$ is not present
in the direct sum). We denote the $[k]$-polystable representations
by $\mathcal{R}_{\Gamma}^{[k]}GL_{n}$ and $\mathcal{X}_{\Gamma}^{[k]}GL_{n}\subset\mathcal{X}_{\Gamma}GL_{n}$
refers to the $[k]$-polystable locus of the character variety.\end{defn}
\begin{rem}
\label{rem:ab-irr-strata} The abelian stratum, i.e, the stratum of
representations factoring as $\Gamma\to\Gamma/[\Gamma,\Gamma]\to G$,
corresponds to the partition $[1^{n}]$ of maximal length $n$ (see
\cite{FNZ}); on the other hand, irreducible representations correspond
to the partition $[n]$ of minimal length $1$. By Lemma \ref{lem:Schur},
this irreducible stratum is also the smooth (and the stable) locus
of the representation varieties $\mathcal{R}_{\Gamma}^{irr}GL_{n}:=\mathcal{R}_{\Gamma}^{[n]}GL_{n}$. 
\end{rem}
The following summarizes the situation and is proved in \cite[Proposition 4.3]{FNZ}. 
\begin{prop}
\label{prop:locally-closed-GLn}Fix $n\in\mathbb{N}$. The character
variety $\mathcal{X}_{\Gamma}GL_{n}$ can be written as a disjoint
union of of locally closed quasi-projective varieties, labelled by
partitions $[k]\in\mathcal{P}_{n}$ 
\[
\mathcal{X}_{\Gamma}GL_{n}=\bigsqcup_{[k]\in\mathcal{P}_{n}}\mathcal{X}_{\Gamma}^{[k]}GL_{n},
\]
where $\mathcal{X}_{\Gamma}^{[k]}GL_{n}$ consists of equivalence
classes of $[k]$-polystable representations. Moreover, $\mathcal{X}_{\Gamma}^{[n]}GL_{n}$
is precisely the open locus $\mathcal{X}_{\Gamma}^{irr}GL_{n}$ of
irreducible classes of representations. 
\end{prop}

\subsection{The free group case}

From now on, we restrict ourselves to the case $\Gamma=F_{r}$, the
free group of rank $r\in\mathbb{N}$, and use the notations 
\[
\mathcal{X}_{r}GL_{n},\quad\quad\mathcal{X}_{r}SL_{n},\quad\quad\mathcal{X}_{r}PGL_{n}
\]
for the corresponding character varieties. We will now define the
$[k]$-polystable loci $\mathcal{X}_{r}^{[k]}SL_{n}\subset\mathcal{X}_{r}SL_{n}$
and $\mathcal{X}_{r}^{[k]}PGL_{n}\subset\mathcal{X}_{r}PGL_{n}$,
and the corresponding stratifications, in analogy to Proposition \ref{prop:locally-closed-GLn}.

Recall the action in (\ref{eq:C-GLnaction}) and (\ref{eq:GLn-PGLn-fibration})
which clearly preserves the polystable type stratification of $GL_{n}$,
so we define 
\[
\mathcal{X}_{r}^{[k]}PGL_{n}:=\mathcal{X}_{r}^{[k]}GL_{n}/\mathcal{R}_{r}\mathbb{C}^{*}=\mathcal{X}_{r}^{[k]}GL_{n}/(\mathbb{C}^{*})^{r}.
\]
The next result is proved in the same way as Proposition \ref{prop:PG-fibration-X}.
We observe that the $E$-polynomials of all strata $\mathcal{X}_{r}^{[k]}GL_{n}$
are $1$-variable polynomials (see \cite{FNZ}). Hence, we use the
notation in \eqref{eq:small-e} for $E$-polynomials in the variable
$x=uv$. 
\begin{prop}
\label{prop:RPG-fibration}Let $F_{r}$ be a free group of rank $r$.
For every $[k]\in\mathcal{P}_{n}$, the fibration 
\[
\mathcal{R}_{r}\mathbb{C}^{*}\to\mathcal{X}_{r}^{[k]}GL_{n}\to\mathcal{X}_{r}^{[k]}PGL_{n}
\]
is special. In particular, $e(\mathcal{X}_{r}^{[k]}GL_{n})=(x-1)^{r}\,e(\mathcal{X}_{r}^{[k]}PGL_{n})$. 
\end{prop}

\subsection{The action of the symmetric group on strictly polystable strata}

For a partition $[k]\in\mathcal{P}_{n}$, we define the $[k]$-stratum
of $\mathcal{X}_{r}SL_{n}$ by restriction of the corresponding one
for $GL_{n}$: 
\[
\mathcal{X}_{r}^{[k]}SL_{n}:=\{\rho\in\mathcal{X}_{r}^{[k]}GL_{n}\,|\,\det\rho=1\},
\]
where the determinant of a (conjugacy class of a) representation is
an element of $\mathcal{R}_{r}\mathbb{C}^{*}$. The action of $\mathcal{R}_{r}\mathbb{C}^{*}$
on $\mathcal{X}_{r}^{[k]}GL_{n}$ does not preserve $\mathcal{X}_{r}^{[k]}SL_{n}$
because of the determinant condition. On the other hand, from the
behaviour of the Zariski topology under closed inclusions and quotients,
the following is clear. 
\begin{prop}
\label{prop:locally-closed-SL+PGL}Fix $n\in\mathbb{N}$. The character
varieties $\mathcal{X}_{r}SL_{n}$ and $\mathcal{X}_{r}PGL_{n}$ can
be written as a disjoint unions of of locally closed quasi-projective
varieties, labelled by partitions $[k]\in\mathcal{P}_{n}$ 
\[
\mathcal{X}_{r}SL_{n}=\bigsqcup_{[k]\in\mathcal{P}_{n}}\mathcal{X}_{r}^{[k]}SL_{n},\quad\quad\mathcal{X}_{r}PGL_{n}=\bigsqcup_{[k]\in\mathcal{P}_{n}}\mathcal{X}_{r}^{[k]}PGL_{n}
\]
Moreover, $\mathcal{X}_{r}^{irr}SL_{n}=\mathcal{X}_{r}^{[n]}SL_{n}$
and $\mathcal{X}_{r}^{irr}PGL_{n}=\mathcal{X}_{r}^{[n]}PGL_{n}$ are
Zariski open. 
\end{prop}
It turns out that the irreducible strata $\mathcal{X}_{r}^{irr}SL_{n}$
and $\mathcal{X}_{r}^{irr}PGL_{n}$ are the most difficult cases to
compare, so we start by studying the ones given by partitions with
at least 2 parts. 
\begin{thm}
\label{thm:SL-2strata}Fix $r$ and $n\in\mathbb{N}$. For a partition
$[k]\in\mathcal{P}_{n}$ with length $s>1$, we have: 
\[
e(\mathcal{X}_{r}^{[k]}GL_{n})=(x-1)^{r}e(\mathcal{X}_{r}^{[k]}SL_{n})
\]
\end{thm}
\begin{proof}
We start with a relation between $E$-polynomials of cartesian products
of irreducible character varieties. Let $s\in\mathbb{N}$ and $\mathbf{n}=(n_{1},\cdots,n_{s})$
be a sequence of $s$ positive integers with $n=\sum_{i=1}^{s}n_{i}$.
Denote: 
\[
\mathsf{X}_{r}^{\mathbf{n}}:=\times_{i=1}^{s}\mathcal{X}_{r}^{irr}GL_{n_{i}},
\]
and 
\[
S\mathsf{X}_{r}^{\mathbf{n}}:=\left\{ \rho=(\rho_{1},\cdots,\rho_{s})\in\mathsf{X}_{r}^{\mathbf{n}}\,|\,\prod_{i=1}^{s}\det\rho_{i}=1\right\} \subset\mathsf{X}_{r}^{\mathbf{n}}.
\]
It is clear that the previous constructions can be carried out in
this setting. For example, letting $J:=(\mathcal{R}_{r}\mathbb{C}^{*})^{s}$,
there is a natural action of $J$ on $\mathsf{X}_{r}^{\mathbf{n}}$:
\[
(\sigma_{1},\cdots,\sigma_{s})\cdot(\rho_{1},\cdots,\rho_{s})=(\sigma_{1}\rho_{1},\cdots,\sigma_{s}\rho_{s}),\quad\quad\sigma\in J=(\mathcal{R}_{r}\mathbb{C}^{*})^{s},\rho_{i}\in\mathcal{X}_{r}^{irr}GL_{n_{i}},
\]
since a scalar multiple of an irreducible representation is again
irreducible.

Define the multiplication map $m:J\to\mathcal{R}_{r}\mathbb{C}^{*}\cong(\mathbb{C}^{*})^{r}$
as follows: 
\[
m(\sigma_{1},\cdots,\sigma_{s})=\sigma_{1}^{m_{1}}\cdots\sigma_{s}^{m_{s}},
\]
where $m_{i}:=n_{i}/d$ with $d=gcd(n_{1},\cdots,n_{s})$ (a power
of a representation into $\mathbb{C}^{*}$ is just the power of every
generator). It turns out that 
\[
H:=\ker m=\{(\sigma_{1},\cdots,\sigma_{s})\in(\mathcal{R}_{r}\mathbb{C}^{*})^{s}\,|\,\sigma_{1}^{m_{1}}\cdots\sigma_{s}^{m_{s}}=1\},
\]
is abelian, connected (because $s>1$, and the $m_{i}$ are coprime)
and reductive. Then, it follows that $H$ is isomorphic to an algebraic
torus of the appropriate dimension, $H\cong(\mathbb{C}^{*})^{r(s-1)},$
since $J=(\mathcal{R}_{r}\mathbb{C}^{*})^{s}\cong(\mathbb{C}^{*})^{rs}$.
This allows us to obtain a diagram of algebraic fibrations (vertical
arrows): 
\[
\begin{array}{ccccc}
H & \subset & (\mathcal{R}_{r}\mathbb{C}^{*})^{s} & \stackrel{m}{\to} & \mathcal{R}_{r}\mathbb{C}^{*}\\
\downarrow &  & \downarrow\\
S\mathsf{X}_{r}^{\mathbf{n}} & \subset & \mathsf{X}_{r}^{\mathbf{n}}\\
\downarrow &  & \downarrow\\
S\mathsf{X}_{r}^{\mathbf{n}}/H & = & \mathsf{X}_{r}^{\mathbf{n}}/J,
\end{array}
\]
where both vertical fibrations (the left fibration being the restriction
to $S\mathsf{X}_{r}^{\mathbf{n}}$) are special.

Now we note that, using an action of a finite group, we can obtain
all strata of the stratification of $\mathcal{X}_{r}^{[k]}GL_{n}$.
To be concrete, denote by $Q_{n}$ the finite set: 
\[
Q_{n}:=\{\mathbf{n}=(n_{1},\cdots,n_{s})\in\mathbb{N}^{s}\ |\ \sum_{i=1}^{s}n_{i}=n\}.
\]
To every element of $Q_{n}$ we associate a unique partition of $n$
as follows: 
\begin{eqnarray*}
p:Q_{n} & \to & \mathcal{P}_{n}\\
\mathbf{n}=(n_{1},\cdots,n_{s}) & \mapsto & [k]:=[1^{k_{1}}\cdots n^{k_{n}}]
\end{eqnarray*}
where $k_{j}$ is the number of entries in the sequence $\mathbf{n}$
equal to $j$. For every $\mathbf{n}\in Q_{n}$, let $S_{\mathbf{n}}:=S_{[k]}\subset S_{n}$
be the subgroup of the symmetric group defined by this partition $[k]=p(\mathbf{n})$.
Moreover, we have isomorphisms of varieties: 
\[
\mathsf{X}_{r}^{\mathbf{n}}/S_{\mathbf{n}}\cong\mathcal{X}_{r}^{[k]}GL_{n},\quad\quad S\mathsf{X}_{r}^{\mathbf{n}}/S_{\mathbf{n}}\cong\mathcal{X}_{r}^{[k]}SL_{n},
\]
as can be easily checked. Now, for every $\mathbf{n}\in Q_{n}$, and
taking $(\mathsf{X}_{r}^{\mathbf{n}}/J)/S_{\mathbf{n}}$ at the bottom,
the above diagram becomes a special algebraic $S_{\mathbf{n}}$-fibration
(with trivial action on the bottom and the top row is still the same),
so we can apply Theorem \ref{thm:W-fibration} and Corollary \ref{cor:special-fibration}
to both vertical fibrations to obtain: 
\[
e^{S_{\mathbf{n}}}(\mathsf{X}_{r}^{\mathbf{n}})=e((\mathsf{X}_{r}^{\mathbf{n}}/J)/S_{\mathbf{n}})\,e^{S_{\mathbf{n}}}(J),\quad\quad e^{S_{\mathbf{n}}}(S\mathsf{X}_{r}^{\mathbf{n}})=e((\mathsf{X}_{r}^{\mathbf{n}}/J)/S_{\mathbf{n}})\,e^{S_{\mathbf{n}}}(H).
\]
and to the top horizontal fibration (which is also special) to get
\[
e^{S_{\mathbf{n}}}(J)=e^{S_{\mathbf{n}}}(H)\,(x-1)^{r},
\]
and putting the formulae together: 
\[
e^{S_{\mathbf{n}}}(\mathsf{X}_{r}^{\mathbf{n}})=e((\mathsf{X}_{r}^{\mathbf{n}}/J)/S_{\mathbf{n}})\,e^{S_{\mathbf{n}}}(H)\,(x-1)^{r}=e^{S_{\mathbf{n}}}(S\mathsf{X}_{r}^{\mathbf{n}})\,(x-1)^{r}
\]
Finally, we just need to take the invariant parts of the equivariant
polynomials: 
\[
e(\mathcal{X}_{r}^{[k]}GL_{n})=e(\mathsf{X}_{r}^{\mathbf{n}}/S_{\mathbf{n}})=e(S\mathsf{X}_{r}^{\mathbf{n}}/S_{\mathbf{n}})\,(x-1)^{r},
\]
whenever $[k]=p(\mathbf{n})\in\mathcal{P}_{n}$ (noting that $(x-1)^{r}=(x-1)^{r}T\in R(S_{\mathbf{n}})[x]$,
where $T$ is the trivial $S_{\mathbf{n}}$ one-dimensional representation,
as elements of the ring $R(S_{\mathbf{n}})[x]$). \end{proof}
\begin{rem}
Even though the above formulas prove these $E$-polynomials are only
functions of $x=uv$, one can not conclude that strata $\mathcal{X}_{r}^{[k]}SL_{n}$
are of Hodge-Tate type without further arguments (see Remark \ref{rem:Hodge-Tate-open}). 
\end{rem}

\subsection{Proof of the main Theorem}

We now outline the proof of the main theorem, which relies crucially
on the following. 
\begin{thm}
\label{thm:irred-Zn-action} The central action of $\mathbb{Z}_{n}^{r}$
on $\mathcal{X}_{r}^{irr}SL_{n}$ giving the quotient map 
\[
\mathcal{X}_{r}^{irr}SL_{n}\to\mathcal{X}_{r}^{irr}PGL_{n}
\]
induces an isomorphism of mixed Hodge structures $H^{*}(\mathcal{X}_{r}^{irr}SL_{n})\cong H^{*}(\mathcal{X}_{r}^{irr}PGL_{n}).$ 
\end{thm}
The proof of Theorem \ref{thm:irred-Zn-action} is delayed to Section
5. We will use differential geometric techniques, taking advantage
of the fact that $\mathcal{X}_{r}^{irr}SL_{n}$ is a smooth variety
and $\mathcal{X}_{r}^{irr}PGL_{n}$ is an orbifold (see \cite{FL2,Sik}).

Assuming Theorem \ref{thm:irred-Zn-action}, we can now prove the
main result of the article. 
\begin{thm}
\label{thm:Main} Let $\Gamma=F_{r}$. Then, for all $[k]\in\mathcal{P}_{n}$,
we have $e(\mathcal{X}_{r}^{[k]}SL_{n})=e(\mathcal{X}_{r}^{[k]}PGL_{n})$.
Consequently, $e(\mathcal{X}_{r}SL_{n})=e(\mathcal{X}_{r}PGL_{n})$.\end{thm}
\begin{proof}
From Theorem \ref{thm:irred-Zn-action}, we have $H^{*}(\mathcal{X}_{r}^{irr}SL_{n})\cong H^{*}(\mathcal{X}_{r}^{irr}PGL_{n})$
as mixed Hodge structures, so that their $E$-polynomials coincide.
For any other stratum $[k]=[1^{k_{1}}\cdots n^{k_{n}}]$, which has
more than one part, the equality $e(\mathcal{X}_{r}^{[k]}SL_{n})=e(\mathcal{X}_{r}^{[k]}PGL_{n})$
follows from Theorem~\ref{thm:SL-2strata}. Finally, the last statement
follows from the additivity of the $E$-polynomial applied to the
locally-closed stratifications of Proposition \ref{prop:locally-closed-SL+PGL}. 
\end{proof}
Noting that $SL_{n}$ and $PGL_{n}$ are Langlands dual groups, and
that our proof seems well adapted to more general actions of finite
subgroups, we put forward the following conjecture (answered here
in the case $G=SL_{n}$), and plan to address the general statement
in a future work. 
\begin{conjecture}
\label{conj:Langlands-dual}Let $\Gamma=F_{r}$ and $G$, $G^{L}$
be complex reductive Langlands dual groups. Then both $\mathcal{X}_{r}G$
and $\mathcal{X}_{r}G^{L}$ are of Hodge-Tate type, and $e(\mathcal{X}_{r}G)=e(\mathcal{X}_{r}G^{L})$. \end{conjecture}
\begin{rem}
The part of the conjecture claiming Hodge-Tate type is still largely
open, even for $G=SL_{n}$ (see also Remark \ref{rem:Hodge-Tate-open}).
To the best of our knowledge, the only free group character varieties
that are known to be balanced are for $G=SL_{n}$ and $G=PGL_{n}$
with $n=2,3$ (see \cite{LM}). 
\end{rem}

\subsection{Katz's theorem and the case $n$ odd}

When $n$ is odd, there is an alternative method to prove that the
Serre polynomials of $\mathcal{X}_{r}SL_{n}$ and $\mathcal{X}_{r}PGL_{n}$
coincide. The argument was mentioned in \cite[Remark 9]{LM}, and
we detail it here, for convenience. Denote by $\mathbb{F}_{q}$ a
finite field with $q$ elements and characteristic $p$, so that $q=p^{s}$,
for some $s\in\mathbb{N}$. A scheme $X$, defined over $\mathbb{Z}$,
is called of \emph{polynomial type} if there is a polynomial $C_{X}(t)\in\mathbb{Z}[t]$
(called the \emph{counting polynomial} for $X$) such that the number
of $\mathbb{F}_{q^{s}}$ points of $X$ is precisely 
\[
|X/\mathbb{F}_{q^{s}}|=C_{X}(q^{s}),
\]
for every $s$ and almost every prime $p$. In \cite[Appendix]{HRV}
(see also \cite{BH}) N. Katz showed that if such a scheme $X$ is
of polynomial type, with counting polynomial $C_{X}$, then the Serre
polynomial of the complex variety $X(\mathbb{C})=X\otimes_{\mathbb{Z}}\mathbb{C}$
coincides with the counting polynomial: 
\[
E_{x}(X(\mathbb{C}))=C_{X(\mathbb{C})}(x).
\]
To apply this to our character varieties, note that the natural surjective
morphism of algebraic groups 
\[
SL_{n}(\mathbb{F}_{p})\to PGL_{n}(\mathbb{F}_{p}),
\]
has as kernel the scalar matrices $aI$ of determinant $1$, so that
$a^{n}=1$. By Dirichlet's Theorem on arithmetic progressions, for
a fixed $n$ odd, there exists an infinite number of primes such that
$(n,p-1)=1$, in which cases there are no non-trivial roots of unity,
so that $SL_{n}(\mathbb{F}_{p})\simeq PGL_{n}(\mathbb{F}_{p})$. This
implies that the representation spaces $\mathcal{R}_{r}SL_{n}(\mathbb{F}_{p})$
and $\mathcal{R}_{r}PGL_{n}(\mathbb{F}_{p})$ are in bijective correspondence,
and the same holds for the number of points of the character varieties
over $\mathbb{F}_{q}$: 
\[
\left|\mathcal{X}_{r}SL_{n}(\mathbb{F}_{p})\right|=\left|\mathcal{X}_{r}PGL_{n}(\mathbb{F}_{p})\right|.
\]
In \cite[Corollary 2.6]{MR} it was shown that $PGL_{n}$-character
varieties are of polynomial type. Therefore, by this result, the $SL_{n}$-character
varieties are also polynomial-count, for $n$ odd, with the same counting
polynomial.

\section{The Irreducible Case}

In this section we prove that $E(\mathcal{X}_{r}^{irr}SL_{n})=E(\mathcal{X}_{r}^{irr}PGL_{n})$
for all $r,n\geq1$, as stated in Theorem~\ref{thm:irred-Zn-action},
thus completing the proof of the main result. Our methods are geometric
in the sense that we mainly use complex algebraic and differential
geometry. In particular, we will use the compact versions of all the
character varieties that we have defined before.

\subsection{Compact representation spaces and their irreducible subspaces}

Consider the compact groups $U(n)$, $SU(n)$ and $PU(n)$, which
are related through the fibrations: 
\[
\begin{array}{ccccc}
SU(n) & \to & U(n) & \to & U(1)=S^{1}\\
S^{1} & \to & U(n) & \to & PU(n),
\end{array}
\]
so that $PU(n)\cong U(n)/S^{1}\cong SU(n)/\mathbb{Z}_{n}$, where
we identify the cyclic group with the group of $n^{th}$ roots of
unity 
\[
\mathbb{Z}_{n}=\{e^{\frac{2\pi ik}{n}}\,:\,k\in\mathbb{Z}\}\subset S^{1},
\]
and with the center of $SU(n)$. For $r\in\mathbb{N}$, consider the
\emph{compact representation spaces}, 
\[
U_{r,n}:=\hom(F_{r},U(n))\cong U(n)^{r},
\]
and similarly $S_{r,n}:=\hom(F_{r},SU(n))\cong SU(n)^{r}$ and $P_{r,n}:=\hom(F_{r},PU(n))\cong PU(n)^{r}$,
where the isomorphisms are obtained when fixing a set of $r$ generators
of the free group $F_{r}$. Because we are dealing with the free group,
representations can be multiplied componentwise: all spaces $U_{r,n}$,
$S_{r,n}$ and $P_{r,n}$ are in fact Lie groups. The first fibration
above defines a determinant map, which is actually a homomorphism
of groups: 
\[
\det:U_{r,n}\to U_{r,1},
\]
whose kernel is precisely $S_{r,n}$.

Now consider the \emph{irreducible representation} spaces, 
\[
U_{r,n}^{*}:=\hom^{irr}(F_{r},U(n)),\quad\quad S_{r,n}^{*}:=\hom^{irr}(F_{r},SU(n)),\quad\quad P_{r,n}^{*}:=\hom^{irr}(F_{r},PU(n)),
\]
which are open subsets in the compact representation spaces: 
\[
U_{r,n}^{*}\subset U_{r,n},\quad\quad SU_{r,n}^{*}\subset SU_{r,n},\quad\quad PU_{r,n}^{*}\subset PU_{r,n}.
\]
Also note that $SU_{r,n}^{*}=U_{r,n}^{*}\cap SU_{r,n}$. These irreducible
subspaces $U_{r,n}^{*},SU_{r,n}^{*},PU_{r,n}^{*}$, are not Lie groups,
but we observe the following straightforward properties of the natural
multiplication action of $U_{r,1}$ on $U_{r,n}$: 
\begin{equation}
\sigma\cdot\rho=(\sigma_{1},\cdots,\sigma_{r})\cdot(\rho_{1},\cdots,\rho_{r}):=(\sigma_{1}\rho_{1},\cdots,\sigma_{r}\rho_{r}),\quad\quad\sigma_{i}\in U(1),\ \rho_{i}\in U(n),\label{eq:gamma}
\end{equation}
where $\rho_{i}=\rho(\gamma_{i})$ and $\sigma_{i}=\sigma(\gamma_{i})$
for $\gamma_{1},\cdots,\gamma_{r}$ the fixed chosen generators of
$F_{r}$. 
\begin{prop}
The action of $U_{r,1}$ on $U_{r,n}$ is such that: \\
 (i) the subspace $U_{r,n}^{*}$ is preserved under the action;\\
 (ii) $PU_{r,n}$ and $PU_{r,n}^{*}$ are, respectively, the orbit
spaces of $U_{r,n}$ and $U_{r,n}^{*}$ under $U_{r,1}$;\\
 (iii) $SU_{r,n}^{*}=\det^{-1}(1)$ for the restriction $\det:U_{r,n}^{*}\to U_{r,1}$,
where $1\in U_{r,1}$ denotes the trivial one dimensional representation. 
\end{prop}
Now define: 
\[
C_{r,n}:=\hom(F_{r},\mathbb{Z}_{n})\subset\hom(F_{r},U(1))=U_{r,1},
\]
so that $C_{r,n}\cong(\mathbb{Z}_{n})^{r}$ is a subgroup of $U_{r,n}$
and can be identified with the center of $SU_{r,n}$. In the same
way as $PU(n)=SU(n)/\mathbb{Z}_{n}$, we can also get $PU_{r,n}$
and $PU_{r,n}^{*}$ as finite quotients of $SU_{r,n}$ and $SU_{r,n}^{*}$:
\[
PU_{r,n}=SU_{r,n}/C_{r,n},\quad\quad PU_{r,n}^{*}=SU_{r,n}^{*}/C_{r,n}.
\]

\subsection{Central action on representation spaces}

We now define a \emph{stratification by polystable type} of $U_{r,n}$,
$SU_{r,n}$ and $PU_{r,n}$ in complete analogy with the stratifications
in \cite[Section 4.1]{FNZ} (for $GL_{n}$) and in Section \ref{section:stratifications}
above (for $PGL_{n}$ and $SL_{n}$).

Given a partition $[k]=[1^{k_{1}}\cdots j^{k_{j}}\cdots n^{k_{n}}]$,
we say that $\rho\in U_{r,n}$ is of type $[k]$ if $\rho$ is conjugated
to $\bigoplus_{j=1}^{n}\rho_{j}$, where each $\rho_{j}$ is a direct
sum of $k_{j}$ \emph{irreducible} representations of $U_{r,j}^{*}$,
for each $j=1,\cdots,n$. We denote representations of type $[k]$
by $U_{r}^{[k]}\subset U_{r,n}$ and let: 
\[
SU_{r}^{[k]}=U_{r}^{[k]}\cap SU_{r,n},\quad\quad PU_{r}^{[k]}=U_{r}^{[k]}/U_{r,1}\subset PU_{r,n}.
\]
Note that all strata are locally closed, and that the irreducible
strata $[k]=[n]$ are the only ones that are open (in the respective
representation spaces), corresponding to the partition into one single
part. 
\begin{prop}
\label{prop:stratification-type}Fix $n\in\mathbb{N}$. The representation
spaces $U_{r,n}$, $SU_{r,n}$ and $PU_{r,n}$ can be written as disjoint
unions, labelled by partitions $[k]\in\mathcal{P}_{n}$, of locally
closed submanifolds: 
\[
U_{r,n}=\bigsqcup_{[k]\in\mathcal{P}_{n}}U_{r}^{[k]},\quad\quad SU_{r,n}=\bigsqcup_{[k]\in\mathcal{P}_{n}}SU_{r}^{[k]},\quad\quad PU_{r,n}=\bigsqcup_{[k]\in\mathcal{P}_{n}}PU_{r}^{[k]}.
\]
\end{prop}
\begin{proof}
The proof for $U_{n,r}$ is analogous to the proof in \cite[Proposition 4.3]{FNZ}
for $\mathcal{R}_{\Gamma}GL_{n}$, observing that dealing with the
usual Euclidean topology on $U_{n,r}$ works \emph{ipsis verbis} as
with the Zariski topology on $\mathcal{R}_{\Gamma}GL_{n}$. The cases
$SU_{r,n}$ and $PU_{r,n}$ follow as in Proposition \ref{prop:locally-closed-SL+PGL}. 
\end{proof}
Now, we state the main result of this subsection. 
\begin{thm}
\label{thm:trivial-action}The action of $C_{r,n}$ on the compactly
supported cohomology of $SU_{r,n}^{*}=\hom^{irr}(F_{r},SU_{n})$ is
trivial. In particular, the natural quotient under $C_{r,n}$ induces
an isomorphism $H_{c}^{*}(SU_{r,n}^{*})\cong H_{c}^{*}(PU_{r,n}^{*})$. 
\end{thm}
Note that we use compactly supported cohomology as dealing with non-compact
spaces. To prove Theorem \ref{thm:trivial-action}, we need to show
two results: (i) under a open/closed decomposition of a compact space
$X=U\sqcup Y$, the triviality of the action on $2$ spaces implies
the same for the third one; (ii) the actions of $C_{r,n}$ on the
cohomology of every $[k]$-stratum are trivial, for partitions $[k]$
of length $l>1$. (iii) the action on the whole $SU_{r,n}$ is trivial.
We start with (iii), which uses the following standard lemma (see,
for instance, \cite{CFLO}). 
\begin{lem}
\label{lem:interpolating-homotopy}If $J$ is a finite group acting
on a space $X$ and, for every $g\in J$, the induced map $\hat{g}:X\to X$,
$x\mapsto g\cdot x$ is homotopic to $id_{X}$, then $J$ acts trivially
on $H^{*}(X)$. In particular, if $F$ is a finite subgroup of a connected
group $G$ acting on $X$, then $F$ acts trivially on $H^{*}(X)$.%

\end{lem}
Since the action of $C_{r,n}$ is the restriction of the action of
the path connected group $SU(n)^{r}$ acting by left multiplication
the following is clear. 
\begin{cor}
\label{cor:action-on-total}The action of $C_{r,n}\cong(\mathbb{Z}_{n})^{r}\subset SU(n)^{r}$
on $H^{*}(SU_{r,n})=H^{*}(SU(n)^{r})$, is trivial. \end{cor}
\begin{rem}
Although the action of $C_{r,n}$ preserves the stratification, the
argument above cannot be applied to the individual strata $SU_{r,n}^{[k]}$;
indeed, it is not clear what \emph{connected} group could interpolate
the action of the discrete group $C_{r,n}$ on the irreducible stratum.
For example, using the left multiplication by the whole group, the
action of $(A^{-1},A^{-1})\in SU(n)^{2}$ on an irreducible pair $(A,B)\in SU(n)^{2}$
gives the pair $(I,A^{-1}B)$, which clearly belongs to the $[1^{n}]$-stratum.
Other attempts at using smaller groups $H\subset SU(n)$, such as
maximal tori, may still fail for the irreducible stratum. Given this
difficulty, we resort to an argument analogous to the one of Theorem
\ref{thm:SL-2strata}, for strata associated with partitions of length
more than one. \end{rem}
\begin{lem}
\label{lem:action-on-strata}Let $[k]\in\mathcal{P}_{n}$ be a partition
with length $l>1$. Then, the action of $C_{r,n}$ on the cohomology
of $SU_{r}^{[k]}$ is trivial. \end{lem}
\begin{proof}
We first consider cartesian products of irreducible representation
spaces satisfying a determinant condition. Let $s>1$, $\mathbf{n}:=(n_{1},\cdots,n_{s})\in\mathbb{N}^{s}$,
with $n=n_{1}+n_{2}+\cdots+n_{s}$ and denote by 
\[
\mathcal{S}^{\mathbf{n}}:=\left\{ \rho=(\rho_{1},\cdots,\rho_{s})\in\times_{j=1}^{s}U_{r,n_{j}}^{*}\ |\ {\textstyle \prod_{j=1}^{s}}\det\rho_{j}=1\right\} 
\]
which is a smooth manifold. Note that, in contrast to partitions,
$\mathbf{n}:=(n_{1},\cdots,n_{s})\in\mathbb{N}^{s}$ is an ordered
$s$-tuple of elements of $\mathbb{N}$. There is an action of $C_{r,n}$
on $\mathcal{S}^{\mathbf{n}}$ given by: 
\begin{equation}
\sigma\cdot\rho=\sigma\cdot(\rho_{1},\cdots,\rho_{s}):=(\sigma\rho_{1},\cdots,\sigma\rho_{s}),\quad\quad\sigma\in\hom(F_{r},\mathbb{Z}_{n})=C_{r,n},\ \rho_{j}\in U_{r,n_{j}}^{*}.\label{eq:gamma-1}
\end{equation}
It is easy to check the product of determinants condition, so that
$\sigma\cdot\rho\in\mathcal{S}^{\mathbf{n}}$.

Now, assume that the greatest common divisor of all $n_{1},\cdots,n_{s}$
is 1. Then, there is at least one $n_{j}$ prime with $n$, and without
loss of generality, we can take $p:=n_{1}$ prime with $n$. Denote
for simplicity, $q=n-p=\sum_{i=2}^{s}n_{i}>0$. Define the following
elements of $U_{r,1}$: 
\[
\sigma_{t}=(e^{-2\pi it\frac{q}{n}},1,\cdots,1),\;\;\;\;\;\ \tilde{\sigma_{t}}=(e^{2\pi it\frac{p}{n}},1,\cdots,1)
\]
parametrised by $t\in[0,1]$, and where each factor corresponds to
a generator of $F_{r}$. Consider the following homotopy: 
\begin{eqnarray*}
[0,1]\times\mathcal{S}^{\mathbf{n}} & \to & \mathcal{S}^{\mathbf{n}}\\
(t;\rho_{1},\rho_{2}\cdots,\rho_{s}) & \mapsto & (\sigma_{t}\rho_{1},\,\tilde{\sigma_{t}}\rho_{2},\cdots,\,\tilde{\sigma_{t}}\rho_{s}).
\end{eqnarray*}
This is well defined since the product of determinants on the right
side, for the first generator $\gamma_{1}$, is: 
\[
\left(e^{-2\pi it\frac{q}{n}}\right)^{p}\left(e^{2\pi it\frac{p}{n}}\right)^{q}{\textstyle \prod_{j=1}^{s}}\det\rho_{j}(\gamma)=1,
\]
(the representation on the other generators does not change under
this homotopy). Then, for $t=0$ the map $\mathcal{S}^{\mathbf{n}}\to\mathcal{S}^{\mathbf{n}}$
is the identity, and for $t=1$ the map is identified with the action,
on the first generator $\gamma_{1}$, of multiplication by the scalar
$e^{2\pi i\frac{p}{n}}$: 
\[
e^{-2\pi i\frac{q}{n}}=e^{-2\pi i\frac{q-n}{n}}=e^{2\pi i\frac{p}{n}},
\]
which is a primitive $n^{th}$ root of unity. Thus, we have found
a homotopy between the identity and $\sigma^{*}$ for the element
$\sigma=(e^{2\pi i\frac{p}{n}},1,\cdots,1)\in C_{r,n}$ (c.f. Lemma
\ref{lem:interpolating-homotopy}). Since all elements of $C_{r,n}$
are obtained as compositions of elements $\sigma$ of this form (with
non-trivial elements on the other entries), we obtain the necessary
homotopies to apply Lemma \ref{lem:interpolating-homotopy}, and finish
the proof, in this case. If the greatest common divisor of all $n_{1},\cdots,n_{s}$
is $d>1$, then we can use the same map but with $t\in[0,\frac{1}{d}]$
(since there is some $p=n_{j}$ that verifies now $(\frac{p}{d},n)=1$).

So, the action of $\mathbb{Z}_{n}^{r}$ is trivial on the cohomology
of $\mathcal{S}^{\mathbf{n}}$. Finally, let $[k]$ be the partition
determined by the tuple $\mathbf{n}=(n_{1},\cdots,n_{s})\in\mathbb{N}^{s}$.
Observe that the length of $[k]$ is $s$. Then, we note that 
\[
SU_{r}^{[k]}=\mathcal{S}^{\mathbf{n}}/S_{[k]},
\]
where $S_{[k]}=\times_{j=1}^{n}S_{k_{j}}\subset S_{n}$ acts by permutation
of the blocks of equal size. Since the action of $C_{r,n}$ commutes
with the one of $S_{[k]}$, we get 
\[
H^{*}(SU_{r}^{[k]})^{C_{r,n}}=\left(H^{*}(\mathcal{S}^{\mathbf{n}})^{S_{[k]}}\right)^{C_{r,n}}=\left(H^{*}(\mathcal{S}^{\mathbf{n}})^{C_{r,n}}\right)^{S_{[k]}}=H^{*}(\mathcal{S}^{\mathbf{n}})^{S_{[k]}}=H^{*}(SU_{r}^{[k]}),
\]
as wanted. 
\end{proof}
For the next Proposition, we let $X$ be a compact Hausdorff topological
space acted by finite group $F$, and we have a closed $F$-invariant
subspace $Y\subset X$, with complement $U:=X\setminus Y$ . Since
we need compactly supported cohomology for the open case, and it coincides
with usual cohomology for $X$ and for $Y$, below we can use $H_{c}^{*}$
uniformly. 
\begin{prop}
\label{prop:snake-cohomology} If $F$ acts trivially on the compactly
supported cohomology of two of the spaces $X$, $Y$ and $U$, then
it acts trivially on the cohomology of the third one. \end{prop}
\begin{proof}
This follows from the 5-lemma applied to the long exact sequence for
cohomology with compact support associated to the decomposition $X=U\sqcup Y$.
More precisely, let $g\in F$, and denote by $g_{Z}^{*}$ the associated
morphisms in the cohomology for $Z=X,Y$ or $U$. Then, we can form
the ladder: 
\[
\begin{array}{ccccccccccccc}
\cdots & \to & H_{c}^{k-1}(X) & \to & H_{c}^{k-1}(Y) & \to & H_{c}^{k}(U) & \to & H_{c}^{k}(X) & \to & H_{c}^{k}(Y) & \to & \cdots\\
 &  & \downarrow g_{X}^{*} &  & \downarrow g_{Y}^{*} &  & \downarrow g_{U}^{*} &  & \downarrow g_{X}^{*} &  & \downarrow g_{Y}^{*}\\
\cdots & \to & H_{c}^{k-1}(X) & \to & H_{c}^{k-1}(Y) & \to & H_{c}^{k}(U) & \to & H_{c}^{k}(X) & \to & H_{c}^{k}(Y) & \to & \cdots.
\end{array}
\]
By hypothesis, two of $g_{X}^{*},g_{Y}^{*}$ and $g_{U}^{*}$ are
isomorphisms. Then, by the 5-lemma, the third map is also an isomorphism.
Since $g_{U}^{*}$ is an isomorphism for every $g\in F$, the action
of $F$ on $H_{c}^{*}(U)$ is trivial. The same argument holds for
the other 2 spaces $X,Y$.
\end{proof}
The stratification $SU_{r,n}=\bigsqcup_{[k]\in\mathcal{P}_{n}}SU_{r}^{[k]}$
has the nice property that the closure of every stratum is a union
of other strata. In fact, for every $a,b\in\mathbb{N}$, the direct
sum of irreducible representations of sizes $a$ and $b$ is in the
closure of the irreducible ones of size $a+b$. This means that the
closure of $SU_{r}^{[k]}$ is the disjoint union of all strata $SU_{r}^{[l]}$
where $[l]$ is obtained by any subdivision of $[k]\in\mathcal{P}_{n}$.
For example, with $n=5$, $\overline{X_{[23]}}=X_{[23]}\sqcup X_{[1^{2}3]}\sqcup X_{[12^{2}]}\sqcup X_{[1^{3}2]}\sqcup X_{[1^{5}]}$,
where we used the abbreviated notation $X_{[k]}:=SU_{r}^{[k]}$.

Now, fix $p\in\{1,\cdots,n\}$ and consider the closed subset of $SU_{r,n}$
defined by 
\[
Y_{p}:=\bigcup_{|[k]|=p}\overline{SU_{r}^{[k]}}=\bigcup_{|[k]|\geq p}SU_{r}^{[k]},
\]
where $|[k]|$ is the length of $[k]\in\mathcal{P}_{n}$. Since each
$SU_{r}^{[k]}$ is the only open stratum in $\overline{SU_{r}^{[k]}}$,
it is easy to see that: $Y_{p}\setminus Y_{p+1}=\bigsqcup_{|[k]|=p}SU_{r}^{[k]},$
and this last union is disjoint. 
\begin{lem}
\label{lem:induction}Fix $n$ and let $2\leq p\leq n-1$. If $C_{r,n}$
acts trivially on the cohomology of $Y_{p+1}$, then the same holds
for $Y_{p}$. Thus, $C_{r,n}$ acts trivially on the cohomology of
$Y_{2}$.\end{lem}
\begin{proof}
Since $p\geq2$, by Lemma \ref{lem:action-on-strata}, $C_{r,n}$
acts trivially on 
\[
H_{c}^{*}(\bigsqcup_{|[k]|=p}SU_{r}^{[k]})=\bigoplus_{|[k]|=p}H_{c}^{*}(SU_{r}^{[k]}).
\]
We can use compactly supported cohomology since each $SU_{r}^{[k]}$
is a quotient of a smooth manifold by a finite group, so that it verifies
Poincaré duality. Assuming $C_{r,n}$ acts trivially on the cohomology
of $Y_{p+1}$, then the same holds for $Y_{p}=Y_{p+1}\sqcup(\sqcup_{|[k]|=p}SU_{r}^{[k]})$,
by Proposition \ref{prop:snake-cohomology}. Noting that $Y_{n}=SU_{r}^{[1^{n}]}$
(already a closed stratum) we have that $C_{r,n}$ acts trivially
on $H^{*}(Y_{n})$ again by Lemma \ref{lem:action-on-strata}. So,
the last sentence follows by finite induction. 
\end{proof}
Finally, Theorem \ref{thm:trivial-action} follows by another application
Proposition \ref{prop:snake-cohomology} and Corollary \ref{cor:action-on-total},
to the spaces $X=Y_{1}=SU_{r,n}$, the closed subset $Y=Y_{2}$ and
$U=X\setminus Y_{2}=SU_{r,n}^{*}=SU_{r}^{[n]}$.%

\subsection{Cohomology of quotient spaces}

Now, we consider the action of $PU(n)$ on all three representation
spaces by simultaneous conjugation, and define the compact irreducible
``character varieties'' $\mathcal{X}_{r}^{*}SU_{n}:=SU_{r,n}^{*}/PU(n)$
and $\mathcal{X}_{r}^{*}PU_{n}:=PU_{r,n}^{*}/PU(n)$.

We will need the notion of \emph{equivariant} \emph{cohomology}, which
we now recall. Let $X$ be a topological space endowed with the action
of a topological group $G$, and $p:E_{G}\to B_{G}$ be the universal
principal $G$-bundle, where $B_{G}=E_{G}/G$ is the classifying space
of $G$. The equivariant cohomology of $X$ is defined by (for more
details see, for example, \cite{Br}) 
\[
H_{G}^{*}(X):=H^{*}(X\times_{G}E_{G}),
\]
and note that, if $G$ acts freely on $X$, then 
\[
H_{G}^{*}(X)\cong H^{*}(X/G).
\]

\begin{prop}
\label{prop:SUPUisomorphisms} There are isomorphisms $H^{*}(SU_{r,n}^{*})\cong H^{*}(PU_{r,n}^{*})$,
$H^{*}(\mathcal{X}_{r}^{*}SU_{n})\cong H^{*}(\mathcal{X}_{r}^{*}PU_{n})$
and $H^{*}(\mathcal{X}_{r}SU_{n})\cong H^{*}(\mathcal{X}_{r}PU_{n})$.\end{prop}
\begin{proof}
The action of $C_{r,n}$ on the compactly supported cohomology of
$SU_{r,n}^{*}=\hom^{irr}(F_{r},SU_{n})$ is trivial, by Theorem \ref{thm:trivial-action},
so: 
\[
H_{c}^{*}(SU_{r,n}^{*})=H_{c}^{*}(SU_{r,n}^{*})^{C_{r,n}}=H_{c}^{*}(SU_{r,n}^{*}/C_{r,n})=H_{c}^{*}(PU_{r,n}^{*}).
\]
Since $SU_{r,n}^{*}$ and $PU_{r,n}^{*}$ are either smooth manifolds
or orbifolds (and are orientable), their cohomologies satisfy Poincaré
duality, so the same isomorphism holds for the usual cohomologies.
Now, as in \cite[p.12]{CFLO}, because the quotient map $\pi:SU_{r,n}^{*}\to SU_{r,n}^{*}/C_{r,n}=PU_{r,n}^{*}$
induces an isomorphism in cohomology, and $\pi$ is equivariant with
respect to the conjugation $PU(n)$ action, we have that their equivariant
cohomologies are the same: 
\[
H_{PU(n)}^{*}(SU_{r,n}^{*})\cong H_{PU(n)}^{*}(PU_{r,n}^{*}).
\]
Since the $PU(n)$ action is free on the irreducible representation
spaces, we get 
\[
H^{*}(\mathcal{X}_{r}^{*}SU_{n})\cong H_{PU(n)}^{*}(SU_{r,n}^{*})\cong H_{PU(n)}^{*}(PU_{r,n}^{*})\cong H^{*}(\mathcal{X}_{r}^{*}PU_{n}),
\]
completing the proof of the second isomorphism. This also means that
the action of $C_{r,n}$ on $H^{*}(\mathcal{X}_{r}^{*}SU_{n})$ is
trivial.

Now, for $[k]=[1^{k_{1}}\cdots n^{k_{n}}]\in\mathcal{P}_{n}$, we
can write every stratum as a finite quotient $\mathcal{X}_{r}^{[k]}SU_{n}=(\times_{j=1}^{n}\mathcal{X}_{r}^{*}SU_{k_{j}})/S_{[k]}$
(with $S_{[k]}=S_{k_{1}}\times\cdots\times S_{k_{n}}$), and since
there is a trivial action of $C_{r,n}$ on $H^{*}(\times_{j=1}^{n}\mathcal{X}_{r}^{*}SU_{k_{j}})=\otimes_{j=1}^{n}H^{*}(\mathcal{X}_{r}^{*}SU_{k_{j}})$,
the same also happens for the subspace fixed by $S_{[k]}$: 
\[
H^{*}(\mathcal{X}_{r}^{[k]}SU_{n})=\left[\otimes_{j=1}^{n}H^{*}(\mathcal{X}_{r}^{*}SU_{k_{j}})\right]^{S_{[k]}}.
\]
Now, the stratification $\mathcal{X}_{r}SU_{n}=\bigsqcup_{[k]\in\mathcal{P}_{n}}\mathcal{X}_{r}^{[k]}SU_{n}$
has also the property that the closure of a stratum is a union of
strata. For $p\in\{2,\cdots,n\}$ consider the closed subset: 
\[
X_{p}:=\bigcup_{|[k]|=p}\overline{\mathcal{X}_{r}^{[k]}SU_{n}}=\bigcup_{|[k]|\geq p}\mathcal{X}_{r}^{[k]}SU_{n},
\]
and note that $X_{p}\setminus X_{p+1}=\bigsqcup_{|[k]|=p}\mathcal{X}_{r}^{[k]}SU_{n}$.
Therefore, Lemma \ref{lem:induction} applies to $X_{p}$ in place
of $Y_{p}$ (here, note that $X_{n}=\mathcal{X}_{r}^{[1^{n}]}SU_{n}$
is already a closed stratum), to show that the cohomology of $X_{2}$
has a trivial $C_{r,n}$ action. Finally, the third isomorphism is
shown by applying Proposition \ref{prop:snake-cohomology} to the
triple: $X:=X_{1}=\mathcal{X}_{r}SU_{n}$, $Y:=X_{2}$ and $U:=X_{1}\setminus X_{2}=\mathcal{X}_{r}^{*}SU_{n}=\mathcal{X}_{r}^{[n]}SU_{n}$.
The action of $C_{r,n}$ on the cohomology of $U$ and $Y$ being
trivial, the same holds for $X$. 
\end{proof}
We are finally ready for the completion of the main result. 
\begin{thm}
\label{thm:equality-irr-PGL-SL} There are isomorphisms of mixed Hodge
structures $H^{*}(\mathcal{X}_{r}SL_{n})\cong H^{*}(\mathcal{X}_{r}PGL_{n})$
and $H_{c}^{*}(\mathcal{X}_{r}^{irr}SL_{n})\cong H_{c}^{*}(\mathcal{X}_{r}^{irr}PGL_{n})$.
In particular, the $E$-polynomials of $\mathcal{X}_{r}^{irr}SL_{n}$
and of $\mathcal{X}_{r}^{irr}PGL_{n}$ coincide.\end{thm}
\begin{proof}
For every reductive group $G$, there is a strong deformation retraction
from $\mathcal{X}_{r}G$ to the orbit space $\hom(F_{r},K)/K$, where
$K$ is a maximal compact subgroup of $G$ (see \cite{FL1}), acting
by conjugation. Since the homotopy is defined by the polar decomposition,
in the case $G=SL_{n}$, the strong deformation retraction from $\mathcal{X}_{r}SL_{n}$
to $\mathcal{X}_{r}SU_{n}$ commutes with the action of $C_{r,n}=\hom(F_{r},\mathbb{Z}_{n})$.
Then, from Proposition \ref{prop:SUPUisomorphisms} we get isomorphisms
in usual cohomology: 
\[
H^{*}(\mathcal{X}_{r}SL_{n})\cong H^{*}(\mathcal{X}_{r}SU_{n})\cong H^{*}(\mathcal{X}_{r}PU_{n})\cong H^{*}(\mathcal{X}_{r}PGL_{n}).
\]
Since the quotient $\mathcal{X}_{r}SL_{n}\to\mathcal{X}_{r}PGL_{n}$
is algebraic, the above isomorphism $H^{*}(\mathcal{X}_{r}SL_{n})\cong H^{*}(\mathcal{X}_{r}PGL_{n})$
is an isomorphism of mixed Hodge structures. Moreover, the strong
deformation retraction from $X:=\mathcal{X}_{r}SL_{n}$ to $SU_{r,n}$
restricts to a strong deformation retraction from the strictly polystable
locus $Y:=\bigsqcup_{|[k]|\geq2}\mathcal{X}_{r}^{[k]}SL_{n}$ to $\bigsqcup_{|[k]|\geq2}SU_{n}^{[k]}/PU(n)$
(because the polar decomposition preserves reducible representations,
see \cite{FL1}). Thus, in the same fashion, we obtain isomorphisms
of mixed Hodge structures $H^{*}(Y)\cong H^{*}(\hat{Y})$, where $\hat{Y}:=\bigsqcup_{|[k]|\geq2}\mathcal{X}_{r}^{[k]}PGL_{n}.$

Now, considering the natural open inclusion $j:X\setminus Y=\mathcal{X}_{r}^{irr}SL_{n}\hookrightarrow X$
and the closed inclusion $i:Y\hookrightarrow X=\mathcal{X}_{r}SL_{n}$,
we can consider the long exact sequence: 
\[
\cdots\to H^{k-1}(X)\to H^{k-1}(Y)\to H_{c}^{k}(X\setminus Y)\to H^{k}(X)\to H^{k}(Y)\to\cdots
\]
which comes from the short exact sequence of sheaves $0\to j_{!}j^{*}\underline{\mathbb{C}}\to\underline{\mathbb{C}}\to i_{*}i^{*}\underline{\mathbb{C}}\to0$
where $\underline{\mathbb{C}}$ is the locally constant $\mathbb{C}$-valued
sheaf (a standard reference is \cite{Iv}, or see \cite[pag. 31]{Gor},
for a summarized account). By applying the same sequence to the closed
subvariety $\hat{Y}\subset\hat{X}:=\mathcal{X}_{r}PGL_{n}$, we get
a sequence of isomorphisms: 
\[
\begin{array}{ccccccccccccc}
\cdots & \to & H^{k-1}(X) & \to & H^{k-1}(Y) & \to & H_{c}^{k}(X\setminus Y) & \to & H^{k}(X) & \to & H^{k}(Y) & \to & \cdots\\
 &  & \parallel\wr &  & \parallel\wr &  & \uparrow &  & \parallel\wr &  & \parallel\wr\\
\cdots & \to & H^{k-1}(\hat{X}) & \to & H^{k-1}(\hat{Y}) & \to & H_{c}^{k}(\hat{X}\setminus\hat{Y}) & \to & H^{k}(\hat{X}) & \to & H^{k}(\hat{Y}) & \to & \cdots
\end{array}
\]
which provide, using the 5-lemma, the wanted isomorphisms: %
\[
H_{c}^{k}(\mathcal{X}_{r}^{irr}SL_{n})=H_{c}^{k}(X\setminus Y)\cong H_{c}^{k}(\hat{X}\setminus\hat{Y})=H_{c}^{k}(\mathcal{X}_{r}^{irr}PGL_{n}).
\]
Finally, since the finite quotient $\mathcal{X}_{r}^{irr}SL_{n}\to\mathcal{X}_{r}^{irr}PGL_{n}$
is algebraic, this implies the isomorphism of mixed Hodge structures
on the corresponding compactly supported cohomology groups, and the
equality of $E$-polynomials. 
\end{proof}

\section{Explicit Computations in the free group case}

\label{section:explicit}When $\Gamma=F_{r}$, the free group in $r$
generators, Mozgovoy and Reineke obtained a general formula for the
count of points in $\mathcal{X}_{r}^{irr}GL_{n}$ over finite fields,
showing that these schemes are of polynomial type (see \cite{MR}).
In this section, we explore these formulae in detail and, by using
the \emph{plethystic exponential} correspondence proved in \cite{FNZ},
we provide a finite algorithm to obtain the Serre polynomials and
the Euler characteristics of all the strata $\mathcal{X}_{r}^{[k]}G$,
for all partitions $[k]\in\mathcal{P}_{n}$ and $G=GL_{n}$, $SL_{n}$
or $PGL_{n}$.

\subsection{Serre polynomials of irreducible $GL_{n}$-character varieties $\mathcal{X}_{r}^{irr}G$.}

Let us recall the definition of the plethystic functions, and the
correspondence proved in \cite{FNZ}. Define the \emph{Adams operator}
$\Psi$ as the invertible $\mathbb{Q}$-linear operator acting on
$\mathbb{Q}[x][[t]]$ by $\Psi(x^{i}t^{k}):=\sum_{l\geq1}\frac{x^{li}t^{lk}}{l},$
where $(i,k)\in\mathbb{N}_{0}^{2}\setminus\{(0,0)\}$, with inverse
given by $\Psi^{-1}(x^{j}t^{k})=\sum_{l\geq1}\frac{\mu(l)}{l}x^{jl}t^{kl},$
and $\mu$ is the Möbius function $\mu:\mathbb{N}\to\{0,\pm1\}$ ($\mu(n)=(-1)^{k}$
if $n$ is square free with $k$ primes in its factorization; $\mu(n)=0$
otherwise).

Given a power series $f\in\mathbb{Q}[x][[t]]$, formal in $t$: 
\[
f(x,t)=1+\sum_{n\geq1}f_{n}(x)\,t^{n}\;,
\]
where $f_{n}(x)\in\mathbb{Q}[x]$ are polynomials in $x$, define
the plethystic exponential, $\pexp:\mathbb{Q}[x][[t]]\to1+t\mathbb{Q}[x][[t]]$,
and plethystic logarithm, $\plog$ as: 
\[
\pexp(f):=e^{\Psi(f)},\quad\quad\plog(f):=\Psi^{-1}(\log f).
\]
As established in \cite[Theorem 1.1 and Corollary 1.2]{FNZ}, $GL_{n}$-character
varieties can be expressed in terms of irreducible character varieties
of lower dimension, by means of the plethystic exponential. 
\begin{thm}
\cite[Theorem 1.1]{FNZ} \label{thm:mainFNZ}Let $\Gamma$ be a finitely
presented group. Then, in $\mathbb{Q}[u,v][[t]]$: 
\[
\sum_{n\geq0}E(\mathcal{X}_{\Gamma}GL_{n};u.v)\,t^{n}=\pexp\left(\sum_{n\geq1}E(\mathcal{X}_{\Gamma}^{irr}GL_{n};u,v)\,t^{n}\right)\;.
\]

\end{thm}
Using results of \cite{MR} for the character varieties of the free
group $F_{r}$, this relationship can be made explicit, in terms of
partitions $[k]=[1^{k_{1}}\cdots d^{k_{d}}]\in\mathcal{P}_{d}$, with
length denoted by $|[k]|=k_{1}+\cdots+k_{d}$. Let $\binom{m}{k_{1},\cdots,k_{d}}=m!(k_{1}!\cdots k_{d}!)^{-1}$
be the multinomial coefficients. 
\begin{prop}
\label{prop:Bnformula} Let $r,n\geq2$. The $E$-polynomials of the
irreducible character varieties $B_{n}^{r}(x):=e(\mathcal{X}_{r}^{irr}GL_{n})$
are explicitly given by: 
\[
B_{n}^{r}(x)=(x-1)\sum_{d|n}\frac{\mu(n/d)}{n/d}\,\sum_{[k]\in\mathcal{P}_{d}}\frac{(-1)^{|[k]|}}{|[k]|}\binom{|[k]|}{k_{1},\cdots,k_{d}}\prod_{j=1}^{d}b_{j}(x^{n/d})^{k_{j}}x^{\frac{n(r-1)k_{j}}{d}\binom{j}{2}}\;,
\]
where $b_{j}(x)$ are given by $F^{-1}(t)=1+\sum_{n\geq1}b_{n}t^{n},$
for the series: 
\begin{equation}
F(t)=1+\sum_{n\geq1}\big((x-1)(x^{2}-1)\ldots(x^{n}-1)\big)^{r-1}\,t^{n}.\label{eq:F(t)}
\end{equation}
\end{prop}
\begin{proof}
As mentioned above, the varieties $\mathcal{X}_{r}^{irr}GL_{n}$ are
of polynomial type by \cite[Thm. 1.1]{MR}. So, by Katz's theorem
\cite[Appendix]{HRV}, $B_{n}^{r}(x):=e_{x}(\mathcal{X}_{r}^{irr}GL_{n})$
is obtained by replacing $q$ by $x$ in the counting polynomial $P_{n}(q):=|\mathcal{X}_{r}^{irr}GL_{n}/\mathbb{F}_{q}|$
which in \cite[Theorem 1.2]{MR} is shown to have generating series:
\[
\sum_{n\geq1}B_{n}^{r}(x)\,t^{n}=\sum_{n\geq1}P_{n}(x)\,t^{n}=(1-x)\plog(S\circ F^{-1}(t))\;,
\]
with $F(t)$ as in \eqref{eq:F(t)}, and $S$ is a $\mathbb{Q}[x]$-linear
shift operator defined on $\mathbb{Q}[x][[t]]$ by $S(t)=t$, and
$S(t^{n}):=x^{(r-1)\binom{n}{2}}t^{n},$ for $n\geq2$. Hence, we
get 
\[
S\circ F^{-1}(t)=1+\sum_{n\geq1}b_{n}(x)\,x^{(r-1)\binom{n}{2}}\,t^{n}\;.
\]
So, the Proposition follows from Lemma \ref{lem:plethystic-Log} below,
using 
\[
f_{n}(x)=b_{n}(x)\,x^{(r-1)\binom{n}{2}}\;
\]
since, for a partition $[k]\in\mathcal{P}_{d}$, we have $f_{j}(x^{n/d})^{k_{j}}=b_{j}(x^{n/d})^{k_{j}}\,x^{\frac{n}{d}(r-1)k_{j}\binom{j}{2}}$,
$j=1,\ldots,d$, as wanted. 
\end{proof}
To complete the proof of Proposition \ref{prop:Bnformula}, we need
the following 
\begin{lem}
\label{lem:plethystic-Log} Given $f_{n}(x)\in\mathbb{Q}[x]$, $n\in\mathbb{N}$,
the coefficient of $t^{n}$ in $\plog(1+\sum_{n\geq1}f_{n}(x)\,t^{n})$
is 
\[
\sum_{d|n}\sum_{[k]\in\mathcal{P}_{d}}\frac{\mu(n/d)}{n/d}\frac{(-1)^{|[k]|-1}}{|[k]|}\binom{|[k]|}{k_{1},\cdots,k_{d}}\prod_{j=1}^{d}f_{j}(x^{n/d})^{k_{j}}\;.
\]
\end{lem}
\begin{proof}
From the $\mathbb{Q}$-linearity of $\Psi^{-1}$, we can write, for
a sequence of polynomials $g_{m}(x)\in\mathbb{Q}[x]$, $m\in\mathbb{N}$,
\begin{eqnarray}
\Psi^{-1}(\sum_{m\geq1}g_{m}(x)t^{m}) & = & \sum_{m\geq1}\Psi^{-1}(g_{m}(x)t^{m})=\nonumber \\
 & = & \sum_{l\geq1}\frac{\mu(l)}{l}g_{1}(x^{l})t^{l}+\sum_{l\geq1}\frac{\mu(l)}{l}g_{2}(x^{l})t^{2l}+\cdots=\nonumber \\
 & = & \sum_{l\geq1}\frac{\mu(l)}{l}\sum_{d\geq1}g_{d}(x^{l})t^{dl}=\nonumber \\
 & = & \sum_{n\geq1}\sum_{d|n}\frac{\mu(n/d)}{n/d}g_{d}(x^{n/d})t^{n}.\label{eq:psi-1}
\end{eqnarray}
Now, from the series development of $\log(1+z)=z-\frac{z^{2}}{2}+\frac{z^{3}}{3}-\cdots$,
we compute, using the multinomial theorem: 
\begin{eqnarray*}
\log(1+\sum_{n\geq1}f_{n}(x)\,t^{n}) & = & \left(\sum_{n\geq1}f_{n}(x)\,t^{n}\right)-\frac{1}{2}\left(\sum_{n\geq1}f_{n}(x)\,t^{n}\right)^{2}+\frac{1}{3}\left(\sum_{n\geq1}f_{n}(x)\,t^{n}\right)^{3}-\cdots\\
 & = & \sum_{m\geq1}\left[\sum_{[k]\in\mathcal{P}_{m}}\frac{(-1)^{|[k]|-1}}{|[k]|}\binom{|[k]|}{k_{1},\cdots,k_{m}}\prod_{j=1}^{m}f_{j}(x)^{k_{j}}\right]t^{m}.
\end{eqnarray*}
Finally, we apply formula \eqref{eq:psi-1} to 
\[
g_{m}(x):=\sum_{[k]\in\mathcal{P}_{m}}\frac{(-1)^{|[k]|-1}}{|[k]|}\binom{|[k]|}{k_{1},\cdots,k_{m}}\prod_{j=1}^{m}f_{j}(x)^{k_{j}}\;,
\]
proving the Lemma. \end{proof}
\begin{rem}
When $n$ is a prime number the formula in Proposition \ref{prop:Bnformula}
simplifies to: 
\[
B_{n}^{r}(x)=(x-1)\left[\frac{b_{1}(x^{n})}{n}+\sum_{[k]\in\mathcal{P}_{n}}\frac{(-1)^{|[k]|}}{|[k]|}\binom{|[k]|}{k_{1},\cdots,k_{n}}\prod_{j=1}^{n}b_{j}(x)^{k_{j}}x^{(r-1)k_{j}\binom{j}{2}}\right]\;.
\]

\end{rem}
Using Proposition \ref{prop:Bnformula}, we can write down explicitly
the $E$-polynomials of $\mathcal{X}_{r}^{irr}GL_{n}$, for any value
of $n$ in a recursive way, the first four cases being as follows. 
\begin{lem}
\label{lemma:Bn(x)} The $E$-polynomials of the irreducible character
varieties $B_{n}^{r}(x)=e(\mathcal{X}_{r}^{irr}GL_{n})$, for $n=1,2,3$
and $4$, using the substitution $s=r-1$, are given by: 
\begin{eqnarray*}
\frac{B_{1}^{r}(x)}{x-1} & = & (x-1)^{r-1}=(x-1)^{s}\;,\\
\frac{B_{2}^{r}(x)}{x-1} & = & \frac{1}{2}b_{1}(x^{2})+\frac{1}{2}b_{1}(x)^{2}-b_{2}(x)x^{s}\;,\\
 & = & (x-1)^{s}\Big((x-1)^{s}x^{s}((x+1)^{s}-1)+\frac{1}{2}(x-1)^{s}-\frac{1}{2}(x+1)^{s}\Big),
\end{eqnarray*}
\\[-7mm] 
\begin{eqnarray*}
\frac{B_{3}^{r}(x)}{x-1} & = & \frac{1}{3}b_{1}(x^{3})-\frac{1}{3}b_{1}(x)^{3}+b_{1}(x)b_{2}(x)x^{s}-b_{3}(x)x^{3s}\\
 & = & (x-1)^{s}\Big(-\frac{1}{3}(x^{2}+x+1)^{s}+(x-1)^{2s}(\frac{1}{3}-x^{s}+x^{s}(x+1)^{s}\\
 &  & +x^{3s}+x^{3s}(x+1)^{s}(x^{2}+x+1)^{s}-2x^{3s}(x+1)^{s})\Big),
\end{eqnarray*}
\\[-8mm] 
\begin{eqnarray*}
\frac{B_{4}^{r}(x)}{x-1} & = & (x-1)^{2s}\Big(\frac{1}{4}(x-1)^{2s}-\frac{1}{4}(x+1)^{2s}+(x^{2}-1)^{s}x^{s}(1-(x+1)^{s})\\
 &  & +\frac{1}{2}(x+1)^{2s}x^{2s}(1-(x^{2}+1)^{s})+\frac{1}{2}(x-1)^{2s}x^{2s}(1-(x+1)^{s})^{2}\\
 &  & -(x-1)^{2s}x^{3s}(-(x+1)^{s}(x^{2}+x+1)^{s}+2(x+1)^{s}-1)\\
 &  & -(x-1)^{2s}x^{6s}(-(x+1)^{s}(x^{2}+x+1)^{s}(x^{3}+x^{2}+x+1)^{s}\\
 &  & +2(x+1)^{s}(x^{2}+x+1)^{s}+(x+1)^{2s}-3(x+1)^{s}+1\big)\Big)\;.
\end{eqnarray*}
\end{lem}
\begin{proof}
We will make the formulae in Proposition \ref{prop:Bnformula} explicit,
by inverting formal power series. If 
\[
{\textstyle F(t)=1+\sum_{n\geq1}a_{n}\,t^{n},\quad\mbox{and }F^{-1}(t)=1+\sum_{n\geq1}b_{n}t^{n}}
\]
are formal inverses, the relation between $a_{n}$ and $b_{n}$ can
be obtained from $\sum_{k\geq0}a_{k}b_{n-k}=0$ ($a_{0}=b_{0}=1$),
recursively (valid for power series over any ring) as 
\begin{equation}
b_{1}=-a_{1}\quad\quad;b_{2}=a_{1}^{2}-a_{2};\quad\quad b_{3}=-a_{1}^{3}+2a_{1}a_{2}-a_{3},\quad\mbox{etc}.\label{eq:recursive}
\end{equation}
Now, employing $a_{n}(x):=\big((x-1)(x^{2}-1)\ldots(x^{n}-1)\big)^{r-1}$
as in (\ref{eq:F(t)}), we get: 
\begin{align*}
b_{1}(x) & =-a_{1}(x)=-(x-1)^{r-1}\;,\\
b_{2}(x) & =a_{1}^{2}(x)-a_{2}(x)=(x-1)^{2r-2}-(x-1)^{r-1}(x^{2}-1)^{r-1}\\
 & =(x-1)^{2r-2}\big(-(x+1)^{r-1}+1\big)\;,\\
b_{3}(x) & =-a_{1}^{3}(x)+2a_{1}(x)\,a_{2}(x)-a_{3}(x)=\\
 & =(x-1)^{3r-3}\big(-(x+1)^{r-1}(x^{2}+x+1)^{r-1}+2(x+1)^{r-1}-1\big),\mbox{ etc}\,,
\end{align*}
which, by substitution in Proposition \ref{prop:Bnformula}, completes
the proof. 
\end{proof}
Recall from \cite[Definition 4.14]{FNZ} the notion of rectangular
partition of $n$: a sequence of non-negative integers $0\leq k_{l,h}\leq n$,
$l,h\in\{1,\cdots,n\}$ satisfying $n=\sum_{l=1}^{n}\sum_{h=1}^{n}l\,h\,k_{l,h}$,
interpreted as a collection of $k_{l,h}$ rectangles of size $l\times h$,
with total area $n$. An element of the set $\mathcal{RP}_{n}$, of
rectangular partitions of $n$, is denoted: 
\[
[[k]]=[(1\times1)^{k_{1,1}}\,(1\times2)^{k_{1,2}}\cdots(1\times n)^{k_{1,n}}\cdots(n\times n)^{k_{n,n}}]\in\mathcal{RP}_{n}\;,
\]
and the ``gluing map'' sends a rectangular partition to a usual
partition 
\begin{eqnarray*}
\pi:\mathcal{RP}_{n} & \to & \mathcal{P}_{n}\\{}
[[k]] & \mapsto & [m]=[1^{m_{1}}\cdots n^{m_{n}}]\quad\mbox{where }m_{l}:=\sum_{h=1}^{n}h\cdot k_{l,h}.
\end{eqnarray*}

With the notion of rectangular partitions, we rephrase \cite[Corollary 1.2]{FNZ}
for the free group $\Gamma=F_{r}$ and $G=GL_{n}$ case. 
\begin{thm}
\cite[Corollary 1.2]{FNZ} \label{thm:individual-strata} The E-polynomial
of the $GL_{n}$-character variety of the free group in $r$ generators
is 
\[
e(\mathcal{X}_{r}GL_{n})=\sum_{[[k]]\in\mathcal{RP}_{n}}\ \prod_{l,h=1}^{n}\frac{B_{l}^{r}(x^{h})^{k_{l,h}}}{k_{l,h}!\,h^{k_{l,h}}},
\]
and the E-polynomial of the stratum corresponding to a partition $[m]\in\mathcal{P}_{n}$
is 
\[
e(\mathcal{X}_{r}^{[m]}GL_{n})={\displaystyle \sum_{[[k]]\in\pi^{-1}[m]}\ \prod_{l,h=1}^{n}\frac{B_{l}^{r}(x^{h})^{k_{l,h}}}{k_{l,h}!\,h^{k_{l,h}}}}\;.
\]

\end{thm}

\subsection{Serre polynomials for $SL_{n}$ and $PGL_{n}$-character varieties
of the free group.}

Recall that, by Proposition \ref{prop:RPG-fibration} and Theorem
\ref{thm:Main}, $E$-polynomials of all strata for the $SL_{n}$
and $PGL_{n}$-character varieties of the free group of rank $r$
can be derived from the corresponding ones for $GL_{n}$, by dividing
out by $(x-1)^{r}$: 
\[
e(\mathcal{X}_{r}^{[k]}SL_{n})=e(\mathcal{X}_{r}^{[k]}PGL_{n})=\frac{e(\mathcal{X}_{r}^{[k]}GL_{n})}{(x-1)^{r}}\;,
\]
for every $[k]\in\mathcal{P}_{n}$, as for the whole character variety.
So, Theorem \ref{thm:individual-strata} allows the computation of
explicit formulae for all $E$-polynomials of $\mathcal{X}_{r}G$,
with $G=GL_{n}$, $SL_{n}$ or $PGL_{n}$.

As examples, this method recovers the polynomial $e(\mathcal{X}_{r}SL_{2})=e(\mathcal{X}_{r}PGL_{2})$
first obtained in \cite{CL}, and the polynomial $e(\mathcal{X}_{r}SL_{3})=e(\mathcal{X}_{r}PGL_{3})$
in \cite[Theorem 1]{LM} (compare also \cite{BH}). We illustrate
the method for $n=3$. 
\begin{thm}
\label{thm:GL3-free}For $s\geq0$, we have: 
\begin{align*}
e(\mathcal{X}_{s+1}SL_{3}) & ={\textstyle \frac{1}{2}}(x-1)^{s+1}(x+1)^{s}x+{\textstyle \frac{1}{3}}(x^{2}+x+1)^{s}x(x+1)+\\
 & +(x-1)^{2s}\Big((x+1)^{s}[x^{3s}(x^{2}+x+1)^{s}+x^{s+1}-2x^{3s}]+x^{3s}-x^{s+1}+{\textstyle \frac{x}{6}}(x+1)\Big)
\end{align*}
\end{thm}
\begin{proof}
By Theorem \ref{thm:individual-strata}, we get for $GL_{3}$ 
\[
e(\mathcal{X}_{r}GL_{3})=B_{3}^{r}(x)+B_{2}^{r}(x)B_{1}^{r}(x)+\frac{B_{1}^{r}(x^{3})}{3}+\frac{B_{1}^{r}(x^{2})B_{1}^{r}(x)}{2}+\frac{B_{1}^{r}(x)^{3}}{6}
\]
where the first term corresponds to $e(\mathcal{X}_{r}^{irr}GL_{3})$,
the second to $e(\mathcal{X}_{r}^{[1\,2]}GL_{3})$ and remaining 3
terms to $e(\mathcal{X}_{r}^{[1^{3}]}GL_{3})$ (see \cite[Figure 4.1]{FNZ}
showing rectangular partitions for $n=3$). Replacing $B_{j}^{r}(x)$,
$j=1,2,3$, by the expressions in Lemma \ref{lemma:Bn(x)} we obtain
the result for $\mathcal{X}_{r}GL_{3}$, and the case of $SL_{3}$
follows immediately. 
\end{proof}
Our method allows explicit expressions for $n=4$ and beyond. In fact,
Theorem \ref{thm:individual-strata} provides the decomposition 
\begin{align}
e(\mathcal{X}_{r}GL_{4}) & =e(\mathcal{X}_{r}^{[4]}GL_{4})+e(\mathcal{X}_{r}^{[1\;3]}GL_{4})+e(\mathcal{X}_{r}^{[2^{2}]}GL_{4})+e(\mathcal{X}_{r}^{[1^{2}\;2]}GL_{4})+e(\mathcal{X}_{r}^{[1^{4}]}GL_{4})\nonumber \\
 & =B_{4}(x)+B_{3}(x)B_{1}(x)+\frac{B_{2}(x)^{2}}{2}+\frac{B_{2}(x^{2})}{2}+\frac{B_{2}(x)B_{1}(x^{2})}{2}+\frac{B_{2}(x)B_{1}(x)^{2}}{2}\nonumber \\
 & +\frac{B_{1}(x^{4})}{4}+\frac{B_{1}(x^{3})B_{1}(x)}{3}+\frac{B_{1}(x^{2})^{2}}{8}+\frac{B_{1}(x^{2})B_{1}(x)^{2}}{4}+\frac{B_{1}(x)^{4}}{24}\;,\label{eq:PGL_4}
\end{align}
as the sum of the 5 strata (which comprise the $11$ terms coming
from the rectangular partitions in \cite[Figure 4.2]{FNZ}). Combining
(\ref{eq:PGL_4}) with formulae in Lemma \ref{lemma:Bn(x)} yields
the computation of the $E$-polynomial for $SL_{4}$ (and then also
$PGL_{4}$). This formula is new. 
\begin{thm}
The E-polynomial of the $SL_{4}$-character variety of $F_{s+1}$
is: 
\begin{eqnarray*}
e(\mathcal{X}_{s+1}SL_{4}) & = & (x-1)^{3s+1}\big[(x+1)^{2s}\frac{x^{2s}}{2}+(x+1)^{s}\big(x^{3s}(x^{2}+x+1)^{s}-2x^{3s}-x^{2s}+\frac{3x^{s}}{2}\big)\big]\\
 & + & (x-1)^{3s+1}\big[x^{3s}+\frac{x^{2s}}{2}-\frac{3x^{s}}{2}+\frac{11}{24}\big]+\frac{1}{24}(x-1)^{3s+3}\\
 & + & (x-1)^{3s}(x+1)^{2s}(-x^{6s}+\frac{x^{2s}}{2})\\
 & + & (x-1)^{3s}(x+1)^{s}x^{6s}\big[(x^{2}+x+1)^{s}(x^{3}+x^{2}+x+1)^{s}-2(x^{2}+x+1)^{s}+3\big]\\
 & + & (x-1)^{3s}(x+1)^{s}\big[x^{3s}\big((x^{2}+x+1)^{s}-2\big)-x^{2s}+\frac{x^{s}}{2}\big]\\
 & + & (x-1)^{3s}(-x^{6s}+x^{3s}+\frac{x^{2s}}{2}-\frac{x^{s}}{2}+\frac{1}{2})\\
 & + & (x-1)^{2s+2}\frac{(x-1)^{s+1}}{4}\\
 & + & (x-1)^{2s+1}\frac{(x+1)^{s}}{2}(-(x+1)^{s}x^{s}+x^{s}-\frac{1}{2})\\
 & + & (x-1)^{2s}(x+1)^{s}\frac{x^{s}}{2}(1-(x+1)^{s})\\
 & + & (x-1)^{s+1}\big[(x+1)\frac{x}{3}(x^{2}+x+1)^{s}+\frac{(x+1)^{2s}}{8}(x^{2}+2x+2)\big]\\
 & + & (x-1)^{s}(x+1)^{2s}\big[\frac{x^{2s+1}}{2}((x^{2}+1)^{s}-1)+\frac{x-1}{4}\big]\\
 & - & \frac{1}{4}(x+1)^{s+1}(x^{2}+1)^{s}+\frac{1}{4}(x^{3}+x^{2}+x+1)^{s+1}\;.
\end{eqnarray*}

\end{thm}
Formulas for all $n$ can be obtained in exactly the same way, from
the combinatorics of rectangular partitions, and the formula for $B_{n}$
in Proposition \eqref{prop:Bnformula}.

\subsection{Irreducibility and Euler characteristics}

It is clear that the above method, combining Proposition \ref{prop:Bnformula},
Lemma \ref{lemma:Bn(x)} and Theorem \ref{thm:individual-strata},
provides the same kind of expressions for $e(\mathcal{X}_{r}SL_{n})=e(\mathcal{X}_{r}PGL_{n})$,
and for every $n\in\mathbb{N}$. Additionally, we can prove irreducibility
and compute all Euler characteristics of $\mathcal{X}_{r}^{[k]}G$
for $G=GL_{n}$, $SL_{n}$ and $PGL_{n}$, and all $[k]\in\mathcal{P}_{n}$.
We start with the $GL_{n}$ case. 
\begin{lem}
\label{dimirr} The degree of the polynomial $B_{n}^{r}(x)$ is 
\[
\deg(B_{n}^{r}(x))=n^{2}(r-1)+1.
\]
\end{lem}
\begin{proof}
From Proposition \ref{prop:Bnformula}, $\deg B_{n}^{r}(x)$ will
be 
\begin{equation}
1+\max_{d|n}\max_{[k]\in\mathcal{P}_{d}}\deg\left(\prod_{j=1}^{d}b_{j}(x^{n/d})^{k_{j}}x^{\frac{n(r-1)k_{j}}{d}\binom{j}{2}}\right),\label{eq:degree}
\end{equation}
where $b_{n}(x)$ are defined by $(1+\sum_{n\geq1}a_{n}(x)\,t^{n})(1+\sum_{n\geq1}b_{n}(x)\,t^{n})=1$,
with $a_{n}(x)=\big((x-1)(x^{2}-1)\ldots(x^{n}-1)\big)^{r-1}$, so
that $\deg a_{i}(x)=\frac{i(i+1)}{2}(r-1)$. By the recursive definition
of $b_{j}(x)$ in terms of the $a_{i}(x)$ (see Equations \eqref{eq:recursive}),
it is easy to see that $b_{j}(x)=-a_{j}(x)+\mbox{lower degree terms}$,
hence $\deg b_{j}(x)=\deg a_{j}(x)=\binom{j+1}{2}(r-1)$. Now, the
maximum in \eqref{eq:degree} is achieved for $d=n$ because partitions
of smaller $n$ have lower degree. Hence: 
\begin{eqnarray*}
\deg B_{n}^{r}(x) & = & 1+\max_{[k]\in\mathcal{P}_{n}}\deg\prod_{j=1}^{n}b_{j}(x)^{k_{j}}x^{(r-1)k_{j}\binom{j}{2}}\\
 & = & {\textstyle 1+\max_{[k]\in\mathcal{P}_{n}}\sum_{j=1}^{n}\Big(\binom{j+1}{2}(r-1)k_{j}+(r-1)k_{j}\binom{j}{2}\Big)}\\
 & = & {\textstyle 1+\max_{[k]\in\mathcal{P}_{n}}\sum_{j=1}^{n}k_{j}j^{2}(r-1)}\;.
\end{eqnarray*}
Using the restriction $n=\sum_{j=1}^{n}j\,k_{j}$, it is clear that
the maximum is achieved for the partition $[n]$, i.e. all $k_{j}=0$
except $k_{n}=1$, so that $\deg B_{n}^{r}(x)=n^{2}(r-1)+1$ as wanted.\end{proof}
\begin{cor}
\label{cor:Euler}Every $[k]$-polystable stratum $\mathcal{X}_{r}^{[k]}GL_{n}$
is an irreducible algebraic variety, and has zero Euler characteristic.\end{cor}
\begin{proof}
The irreducible stratum corresponds to the partition $[n]$ and its
polynomial is given by $e(\mathcal{X}_{r}^{irr}GL_{n})=B_{n}^{r}(x)$
(see Theorem \ref{thm:individual-strata} and Remark \ref{rem:ab-irr-strata}).
The monomial of top degree of $B_{n}^{r}(x)$ is, by the proof of
Lemma \ref{dimirr}, 
\[
(x-1)\frac{\mu(1)}{1}\frac{(-1)^{1}}{1}\binom{1}{1}b_{n}(x)^{1}x^{(r-1)\cdot1\cdot\binom{n}{2}}=(x-1)a_{n}(x)x^{(r-1)\binom{n}{2}}\;,
\]
whose leading coefficient equals the leading coefficient of $a_{n}(x)$,
which is $1$. Therefore, $\mathcal{X}_{r}^{irr}GL_{n}$ is an irreducible
variety. All polystable strata can be expressed as symmetric products
of irreducible strata of lower dimension (see \cite[Proposition 4.5]{FNZ}),
hence all of them are irreducible. The statement about Euler characteristics
follows by substituting $x=1$ in Theorem \ref{thm:individual-strata},
since $B_{n}^{r}(1)=0$, for all $n,r\in\mathbb{N}$ (see Proposition
\ref{prop:Bnformula} and Lemma \ref{lemma:Bn(x)}). 
\end{proof}
The next Corollary is immediate from our constructions. 
\begin{cor}
For every $[k]\in\mathcal{P}_{n}$, the strata $\mathcal{X}_{r}^{[k]}SL_{n}$
and $\mathcal{X}_{r}^{[k]}PGL_{n}$ are irreducible algebraic varieties. 
\end{cor}
Finally, we compute the Euler characteristics of the $SL_{n}$ and
$PGL_{n}$-character varieties of the free group. It turns out that
the only strata contributing are of the form $[d^{n/d}]\in\mathcal{P}_{n}$,
indexed by the divisors $d$ of $n$. 
\begin{prop}
\label{prop:euler} The Euler characteristics of the $SL_{n}$ and
$PGL_{n}$ character varieties of the free group are 
\[
\chi(\mathcal{X}_{r}SL_{n})=\chi(\mathcal{X}_{r}PGL_{n})=\varphi(n)n^{r-2}
\]
where $\varphi(n)$ is Euler's totient function. The Euler characteristics
for the strata of the form $[d^{n/d}]\in\mathcal{P}_{n}$ are 
\[
\chi(\mathcal{X}_{r}^{[d^{n/d}]}SL_{n})=\chi(\mathcal{X}_{r}^{[d^{n/d}]}PGL_{n})=\frac{\mu(d)}{d}n^{r-1}\;,
\]
otherwise $\chi(\mathcal{X}_{r}^{[k]}SL_{n})=0$, where $\mu(n)$
is the arithmetic Möbius function. \end{prop}
\begin{proof}
By Theorems \ref{thm:SL-2strata} and \ref{thm:equality-irr-PGL-SL},
the Euler characteristics for $SL_{n}$ and $PGL_{n}$-character varieties
are equal and can be computed strata by strata, by setting $x=1$
in the $GL_{n}$ polynomials of Proposition \ref{prop:Bnformula}
and Theorem \ref{thm:individual-strata}, and using Proposition \ref{prop:RPG-fibration}:
\begin{equation}
\chi(\mathcal{X}_{r}^{[m]}SL_{n})=\chi(\mathcal{X}_{r}^{[m]}PGL_{n})=\Big[\frac{e(\mathcal{X}_{r}^{[m]}GL_{n})}{(x-1)^{r}}\Big]_{x=1}\;.\label{eq:cancelling}
\end{equation}
To begin, we show the formula: 
\begin{equation}
\Big[\frac{B_{n}^{r}(x)}{(x-1)^{r}}\Big]_{x=1}=\Big[\frac{-1}{(x-1)^{r-1}}\frac{\mu(n)}{n}b_{1}(x^{n})\Big]_{x=1}=\mu(n)n^{r-2}\;.\label{eq:B/(x-1)}
\end{equation}
Indeed, by the recursive definition (\ref{eq:recursive}), the term
$(x-1)^{n(r-1)}$ can be factored out in $a_{n}(x)$ and, hence, the
same applies to every $b_{n}(x)$. Then, for $n=1$ we get $b_{1}(x)=-(x-1)^{r-1}$.
However, for $m>1$, we have $\Big[\frac{b_{m}(x)}{(x-1)^{r-1}}\Big]_{x=1}=0$,
so all terms with $b_{m}(x)$ for $m>1$, in the formula of Proposition
\ref{prop:Bnformula}, disappear from the calculation of the Euler
characteristic. For $b_{1}(x)$ we have the following expression,
for every $t\in\mathbb{N}$: 
\[
b_{1}(x^{s})^{t}=-(x^{s}-1)^{t(r-1)}=(-1)^{t}(x-1)^{t(r-1)}(1+x+\cdots+x^{s-1})^{t(r-1)}\;.
\]
Therefore, if $t>1$, $\Big[\frac{b_{1}(x^{s})^{t}}{(x-1)^{r-1}}\Big]_{x=1}=0$,
while, for $t=1$, we have $\Big[\frac{b_{1}(x^{s})}{(x-1)^{r-1}}\Big]_{x=1}=-s^{r-1}$.

Note that products $b_{i}(x)\cdot b_{j}(x)$, $i\neq j$, necessarily
have a zero of order greater than $r$ in $x=1$. Hence, in Lemma
\ref{lemma:Bn(x)}, after setting $\frac{B_{n}^{r}(x)}{(x-1)^{r}}$,
the only terms which do not vanish are those with $j=1$ and $k_{1}=1$
and $k_{j}=0$ for $j=2,\ldots,d$. Therefore $d=1$ and the only
partition contributing is $[1]\in\mathcal{P}_{1}$, yielding the equality
in Equation \eqref{eq:B/(x-1)}. Similarly, we have $\Big[\frac{B_{n}^{r}(x^{s})}{(x-1)^{r}}\Big]_{x=1}=\mu(n)s^{r}n^{r-2}$,
for every $s\in\mathbb{N}$.

To finish, we look at Theorem \ref{thm:individual-strata}. For rectangular
partitions with two or more blocks of different sizes (i.e. giving
terms of the form $B_{n_{1}}^{r}(x^{s_{1}})\cdot B_{n_{2}}^{r}(x^{s_{2}})$)
$x=1$ is a zero of order greater than $r$, then we get rid of them
in the final calculation. This leaves us to just consider rectangular
partitions $[[k]]$ with a single block $l\times h$ and $k_{l,h}=1$
and the rest $k_{l,h}=0$: they are indexed by divisors of $n$. Let
$[k]\in\mathcal{P}(n)$ be a partition by blocks of just one size
$[k]=[d^{n/d}]$. In the expression of Theorem \ref{thm:individual-strata},
put $l=d$ and $h=n/d$ and have 
\[
\chi(\mathcal{X}_{r}^{[d^{n/d}]}SL_{n})=\Big[\frac{e(\mathcal{X}_{r}^{[d^{n/d}]}GL_{n})}{(x-1)^{r}}\Big]_{x=1}=\Big[\frac{1}{(x-1)^{r}}\frac{d}{n}B_{d}(x^{n/d})\Big]_{x=1}=\frac{\mu(d)}{d}n^{r-1}\;,
\]
while for other strata $[k]\in\mathcal{P}(n)$ with two or more blocks
of different sizes, the Euler characteristic is zero. To get the Euler
characteristic of the full character variety, we sum up over all strata,
this is, over all divisors of $n$: 
\[
\chi(\mathcal{X}_{r}SL_{n})=\sum_{d|n}\Big[\frac{1}{(x-1)^{r}}\frac{d}{n}B_{d}(x^{n/d})\Big]_{x=1}=\sum_{d|n}\frac{\mu(d)}{d}n^{r-1}=\varphi(n)n^{r-2},
\]
by using Möbius inversion formula with the Euler function, $\varphi(n)=\sum_{d|n}\frac{n}{d}\mu(d)$. \end{proof}
\begin{rem}
This result extends \cite[Theorem 1.4]{MR} to the $SL_{n}$ case
and, moreover, calculates the Euler characteristic of all strata,
providing a geometrical meaning to the calculation in terms of the
Euler function. Also, note that the stratum $[1^{n}]$ corresponds
to an abelian character variety, and this computation agrees with
the one in \cite[Corollary 5.16]{FS}. \end{rem}
\begin{example}
(1) As an application of Proposition \ref{prop:euler}, for $n=4$,
we get 
\[
\chi(\mathcal{X}_{r}SL_{4})=\chi(\mathcal{X}_{r}PGL_{4})=2\cdot4^{r-2},
\]
where the only strata that contribute to the Euler characteristic
are: 
\[
\chi(\mathcal{X}_{r}^{[1^{4}]}SL_{4})=\chi(\mathcal{X}_{r}^{[1^{4}]}PGL_{4})=4^{r-1},\quad\quad\chi(\mathcal{X}_{r}^{[2^{2}]}SL_{4})=\chi(\mathcal{X}_{r}^{[2^{2}]}PGL_{4})=-2\cdot4^{r-2}.
\]
(2) In the case when $n$ is a prime number $p$, we similarly get:
\[
\chi(\mathcal{X}_{r}SL_{p})=(p-1)\cdot p^{r-2},
\]
where the only strata contributing to the Euler characteristic are:
\[
\chi(\mathcal{X}_{r}^{[1^{p}]}SL_{p})=p^{r-1},\quad\quad\chi(\mathcal{X}_{r}^{[p]}SL_{p})=-p^{r-2},
\]
and similarly for $\mathcal{X}_{r}PGL_{p}$.\\
 (3) More generally, if $n=p_{1}^{a_{1}}\cdot p_{2}^{a_{2}}\cdots p_{s}^{a_{s}}$
is the prime decomposition of $n$, the only strata contributing to
the Euler characteristic of the $SL_{n}$ and the $PGL_{n}$ character
variety of $F_{r}$ are those of the form $[d^{n/d}]$, where $d=p_{i_{1}}\cdot p_{i_{2}}\cdots p_{i_{t}}$
is square-free since then $\mu(d)\neq0$. These strata are the same
for $n^{\ast}:=p_{1}\cdot p_{2}\cdots p_{s}$ and for all $n$ with
the same prime divisors. Indeed, 
\[
\chi(\mathcal{X}_{r}^{[1^{n}]}SL_{n})=n^{r-1}\;,
\]
\[
\chi(\mathcal{X}_{r}^{[p_{i}^{n/p_{i}}]}SL_{n})=-\frac{1}{p_{i}}n^{r-1}\;,
\]
\[
\chi(\mathcal{X}_{r}^{[(p_{i}p_{j})^{n/p_{i}p_{j}}]}SL_{n})=\frac{1}{p_{i}p_{j}}n^{r-1}\;,\quad\cdots
\]
\[
\chi(\mathcal{X}_{r}^{[(n^{\ast})^{n/n^{\ast}}]}SL_{n})=(-1)^{s}\frac{1}{n^{\ast}}n^{r-1}\;,
\]
with sum 
\[
\prod_{p|n,\,p\;prime}(1-\frac{1}{p})n^{r-1}=\varphi(n)n^{r-2}\;.
\]
\end{example}

\end{document}